\documentclass[11pt]{amsart}

\usepackage{paralist, amsmath, amsthm, amssymb, color, graphicx}
\usepackage[margin=1.4in]{geometry}
\usepackage{hyperref}
\usepackage{mathtools}

\newtheorem{thm}{Theorem}

\newtheorem{prop}[thm]{Proposition}
\newtheorem{lemma}[thm]{Lemma}
\newtheorem{cor}[thm]{Corollary} 
 
\newtheorem{conjecture}[thm]{Conjecture}
\newtheorem{Question}[thm]{Question}

\renewcommand{\Pr}{\mathbb{P}}

\newcommand{\Pdet}{\Pr\!\det}

\newcommand{\fS}{\mathfrak{S}}

\renewcommand{\l}{\left}
\renewcommand{\r}{\right}

\newcommand{\al}{\alpha}
\newcommand{\be}{\beta}
\newcommand{\ga}{\gamma}
\newcommand{\de}{\delta}
\newcommand{\eps}{\epsilon}
\newcommand{\si}{\sigma}
\newcommand{\la}{\lambda}
\newcommand{\om}{\omega}

\newcommand{\ze}{\zeta}

\newcommand{\Om}{\Omega}

\newcommand{\vp}{f}

\newcommand{\Id}{\textrm{Id}}
\renewcommand{\Im}{\textrm{Im}}
\newcommand{\cyc}{\textrm{cycle}}

\newcommand{\ds}{\displaystyle}

\newcommand{\fig}[3]{\begin{figure}[h!]\begin{center}\includegraphics[#1]{#2}\end{center}\caption{#3}\label{fig:#2}\end{figure}}

\newcommand{\CC}{\mathbb{C}}

\newcommand{\bX}{\textbf{E}}

\newcommand{\mA}{\mathcal{A}}

\newcommand{\mM}{\mathcal{R}}
\newcommand{\mP}{\mathcal{P}}
\newcommand{\mS}{\mathcal{S}}
\newcommand{\mT}{\mathcal{T}}
\newcommand{\mW}{\mathcal{W}}
\newcommand{\mX}{\mathcal{X}}
\newcommand{\mY}{\mathcal{Y}}
\newcommand{\mZ}{\mathcal{Z}}

\newcommand{\pp}{\textbf{p}}
\newcommand{\qq}{\textbf{q}}

\renewcommand{\SS}{\textbf{S}}

\title{Some probabilistic trees with algebraic roots}

\author{Olivier Bernardi and Alejandro H. Morales}
\thanks{O.B. acknowledges support from NSF grant DMS-1308441, and ERC ExploreMaps.}
\thanks{A.M. acknowledges support from a CRM-ISM postdoctoral fellowship.}
\date{\today}

\begin{document}
\setcounter{tocdepth}{2}

\begin{abstract}
In this article we consider several probabilistic processes defining random grapha. One of these processes appeared recently in connection with a factorization problem in the symmetric group.
For each of the probabilistic processes, we prove that the probability for the random graph to be a tree has an extremely simple expression, which is independent of most parameters of the problem. This raises many open questions.
\end{abstract}

\maketitle

\section{Introduction: an example} \label{sec:intro}

In this paper we consider several probabilistic processes defining a random graph.   These processes were originally motivated by factorizations problems in the symmetric group investigated in \cite{OB-AM:constellations}. Our main result is a formula, for each of the probabilistic processes, of the probability that the random graph is a tree. This probability formula turns out to be surprisingly simple, and is in particular independent of most parameters of the processes. This is reminiscent of the result of Kenyon and Winkler~\cite{Kenyon:branched-polymers} about 2-dimensional Branched polymers. We conjecture that more results of this type should hold, but we were not able to prove them.

Before describing our results in full details, let us describe one particular case. Let $k$ be a positive integer. Given a tuple  $\SS=(S_1,\ldots,S_{k-1})$ of $k-1$ proper subsets of $[k]:=\{1,2,\ldots,k\}$,  we define $G(\SS)$ as the digraph having vertex set $[k]$ and arc set  $\{a_1,\ldots,a_{k-1}\}$, where the arc $a_i$ has origin $i$ and endpoint the unique integer $j$ in $[k]\setminus S_i$ such that  $\{i+1,i+2,\ldots,j-1\}\subseteq S_i$, with the integers considered cyclically modulo~$k$.  For instance, if $k=7$ and $S_5=\{1,3,6,7\}$, then the arc $a_5$ has origin~5 and endpoint~2. In Figure~\ref{fig:first-example} we have drawn some digraphs $G(\SS)$ in the case $k=3$. We now fix a tuple $\pp=(p_1,\ldots,p_k)$ of non-negative integers, and choose uniformly at random a tuple $\SS=(S_1,\ldots,S_{k-1})$ of $k-1$ proper subsets of $[k]$ such that for all $j\in[k]$ the integer $j$ is contained in exactly $p_j$ of the subsets $S_1,\ldots,S_{k-1}$. This gives a random digraph $G(\SS)$. One of the results proved in this paper is that the probability that $G(\SS)$ is a tree (oriented toward the vertex $k$) is equal to $1-p_k/(k-1)$. This result is unexpectedly simple, especially because it does not depend on the parameters $p_1,\ldots,p_{k-1}$.


\fig{width=\linewidth}{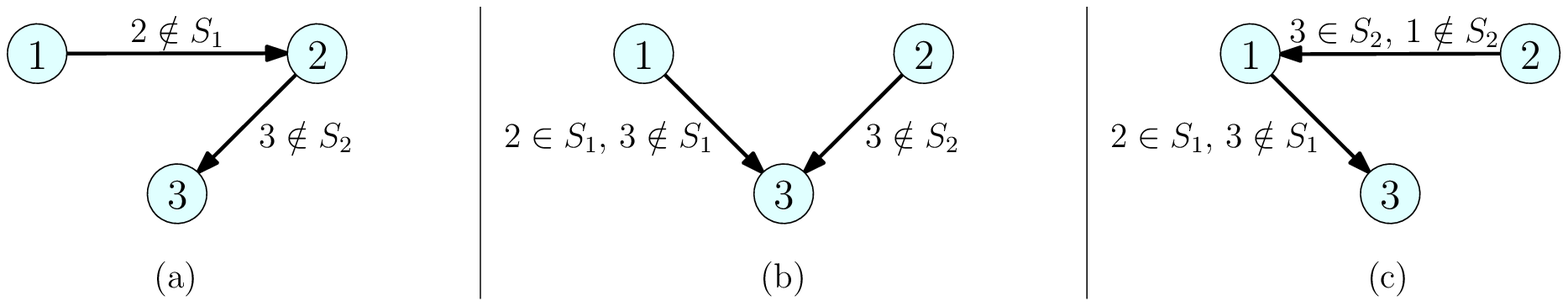}{We represent above the three situations for which the digraph $G(\SS)$ is a tree, with $k=3$ and $\SS=(S_1,S_2)$.}

Let us investigate in more detail the case $k=3$ of the aforementioned result. The situation is represented in Figure~\ref{fig:first-example}. By definition, the tuple $\SS=(S_1,S_2)$ is a pair of proper subsets of $\{1,2,3\}$, and there are three possible trees with vertex set $\{1,2,3\}$ (in general, there are $k^{k-2}$ possible Cayley trees). Let $A$, $B$, and $C$ be respectively the events leading to the trees represented in parts (a), (b), and (c) of Figure~\ref{fig:first-example}. For instance, 
$$C=\{2\in S_1\} ~\cap~ \{3\notin S_1\} ~\cap~ \{3\in S_2\} ~\cap~ \{1\notin S_2\}.$$
Now it is not hard to check that the event $C$ has the same probability as the event 
$$C'=\{2\in S_1\} ~\cap~\{3\notin S_2\}~\cap~ \{3\in S_1\}~\cap~ \{1\notin S_1\}.$$ 
Moreover the events $A$, $B$ and $C'$ are disjoint and their union is
$$A\cup B \cup C'=\big(\{2\notin S_1\} \cup \{2\in S_1 ~\cap~3\notin S_1\} \cup \{2\in S_1~\cap~ 3\in S_1 ~\cap~ 1\notin S_1\}\big) ~\cap~ \{3\notin S_2\}.$$
But since $S_1$ is by definition a \emph{proper} subset of $\{1,2,3\}$, the first condition in the above clause is always satisfied. It follows that 
$$\Pr(A\cup B\cup C)=\Pr(A\cup B\cup C')=\Pr(3\notin S_2)=1-p_3/2,$$
as claimed. Observe that the individual probabilities of the trees represented in Figure~\ref{fig:first-example} \emph{do} depend on the value of $p_1$ and $p_2$, but the sum of these probabilities is independent of $p_1$ and $p_2$. 


The rest of the paper is organized as follows. In Section \ref{sec:results}, we state the main results of the paper.  In Section~\ref{sec:matrix-tree}, we derive a generalization of the matrix-tree theorem tailored to our needs: it allows us to express the probability that our random graphs are trees as the probability of a ``determinant of a matrix of events''. In Section~\ref{sec:computing-determinant} we simplify the matrices of events corresponding to our different random processes. In Section~\ref{sec:computing-determinant}, we compute the determinant of the matrices of events. This computation uses some sign reversing involutions. We conclude in Section~\ref{sec:conj-gamma} with some conjectures about another random process, and some open questions. 


\section{Main results}\label{sec:results}
In this section we fix some notation and state our main results. 
We denote by $|A|$ the cardinality of a set $A$. We denote  $A\uplus B$ the disjoint union of two sets $A,B$. For a positive integer $k$, we denote by $[k]$ the set of integers $\{1,2,\ldots,k\}$. For $i,j\in[k]$, we denote by $]i,j]$ the set of integers $\{i+1,i+2,\ldots,j\}$, where integers are considered cyclically modulo $k$. For instance $]i,i]=\emptyset$, $]i,i+1]=\{i+1\}$, and $]i,i-1]=[k]\setminus\{i\}$.
For an integer $r$ and a tuple $\pp=(p_1,\ldots,p_k)$ of non-negative integers, we denote by  $\mS_{\pp,r}$ the set of tuples $(S_1,\ldots,S_{r})$ such that for all $i\in[r]$, $S_i$ is a subset of $[k]$ and for all $j\in[k]$ the integer $j$ is contained in exactly $p_j$ of the subsets $S_1,\ldots,S_r$. We also denote by $\mM_{\pp,r}$ the set of tuples $(S_1,\ldots,S_{r})\in \mS_{\pp,r}$ such that $S_i\neq [k]$ for all $i\in[r]$.

We now define three ways of associating a digraph to a an element in $\mS_{\pp,r}$, using three mappings $\al,\be,\ga$. The mappings $\al,\be,\ga$ takes as argument an integer $i\in [k]$ and a subset $S\subseteq [k]$ and are defined as follows: 
\begin{itemize}
\item $\al(i,S)=i$ if $S=[k]$, and otherwise $\al(i,S)$ is the integer $j\in[k]$ such that $j\notin S$ but $]i,j-1]\subseteq S$.
\item $\be(i,S)=i$ if $S=[k]$, and otherwise $\be(i,S)$ is the integer $j\in[k]$ such that $j+1\notin S$ but $]i,j]\subseteq S$.
\item $\ga(i,S)=i$ if $S=[k]$, $\ga(i,S)=i-1$ if $i\in S$, and otherwise $\ga(i,S)$ is the integer $j\in[k]$ such that $j+1\notin S$ but $]i,j]\subseteq S$. 
\end{itemize}
The mappings $\al,\be,\ga$ are represented in Figure~\ref{fig:alpha-beta-gamma}.
\fig{width=\linewidth}{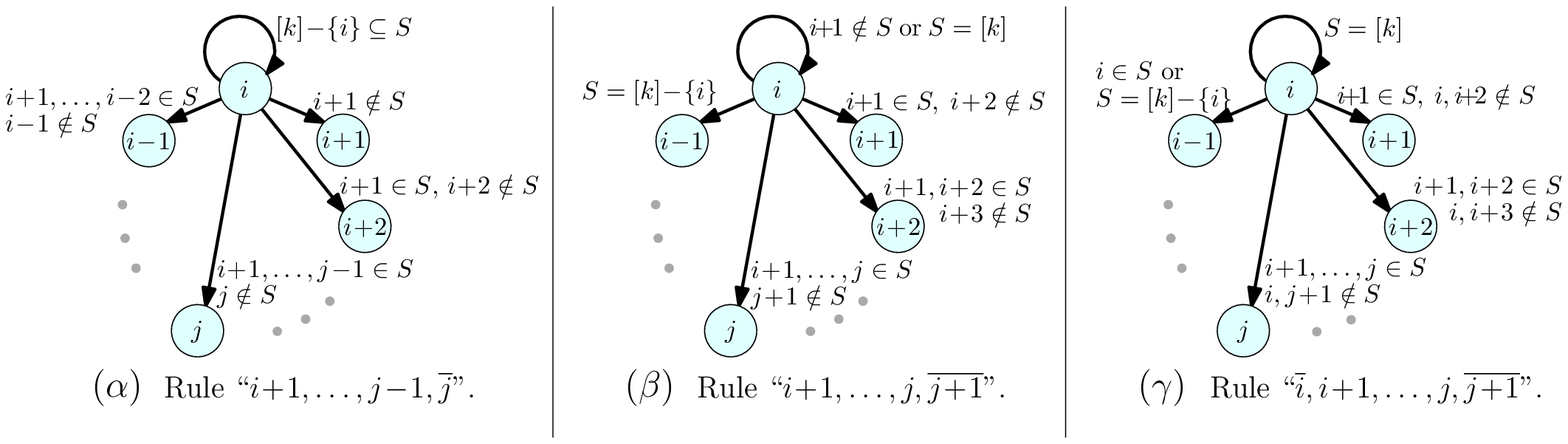}{Rules $\al,\be,\ga$ for creating an arc of the complete graph $K_{k}$ (an informal description of the rule is given between quotation marks where $\overline{i}$ means that $i \not\in S$).}

We now  use the mappings $\al,\be,\ga$ to define digraphs. Let $\SS=(S_1,\ldots,S_r)\in\mS_{\pp,r}$, and let $\vp$ be a surjection from $[k-1]$ to $[r]$. For $\ze\in\{\al,\be,\ga\}$, we define $G_\ze(\SS,\vp)$ to be the digraph with vertex set $[k]$ and arc set $A=\{a_1,\ldots,a_{k-1}\}$ where $a_i=(i,\ze(i,S_{\vp(i)}))$ for all $i\in[k-1]$. For instance, the digraph $G(\SS)$ defined in Section \ref{sec:intro} corresponds to the case $r=k-1$ and $G(\SS)=G_\al(\SS,\Id)$ where $\Id$ is the identity mapping from $[k-1]$ to $[k-1]$. Observe that the graph $G_\ze(\SS,\vp)$ has loops unless  $\SS\in\mM_{\pp,r}$.

We are now interested in the probability that the digraph $G_\ze(\SS,\vp)$ is a tree. Observe that in this case the tree is oriented toward the vertex $k$ (since every vertex in $[k-1]$ has one outgoing arc). Our main result is the following.
\begin{thm}\label{thm:main}
Let $k$ and $r$ be positive integers such that $r<k$, and let $\pp=(p_1,\ldots,p_k)$ be a tuple of non-negative integers.
Let $\SS=(S_1,\ldots,S_r)$ be a uniformly random element of $\mM_{\pp,r}$ (supposing that this set is non-empty), and let $\vp$ be a uniformly random surjection from $[k-1]$ to $[r]$ independent from~$\SS$. For $\ze\in\{\al,\be,\ga\}$, let $P_\ze(\pp,r)$ be the probability that the random digraph $G_\ze(\SS,\vp)$ is a tree. Then 
\begin{itemize}
\item[(a)] $\ds P_\al(\pp,r)=1-p_k/r$. This is equal to the probability that $k\notin S_1$.
\item[(b)] $\ds P_\be(\pp,r)=|\mM_{\qq,r-1}|/|\mM_{\pp,r}|$ where $\qq=(p_1,p_2-1,p_3-1,\ldots,p_k-1)$. This is equal to the probability that $S_1=\{2,3,\ldots,k\}$.
\item[(c)] $\ds P_\ga(\pp,r)=\sum_{j=1}^k|\mM_{\qq^{(j)},r-1}|/|\mM_{\pp,r}|$, where $\qq^{(j)}=(q_1^{(j)},\ldots,q_k^{(j)})$ and for all $i,j\in[k]$, $q_i^{(j)}=p_{i}-1$ if $i\neq j$ and $q_j^{(j)}=p_{j}$. This is equal to the probability that $|S_1|=k-1$.
\end{itemize}
\end{thm}

Observe that the cardinality of $\mM_{\pp,r}$ appearing in Theorem~\ref{thm:main} can be expressed as the coefficient of $x_1^{p_1}x_2^{p_2}\cdots x_k^{p_k}$ in the polynomial $\left(\prod_{i=1}^k(x_i+1)-\prod_{i=1}^kx_i\right)^r$. 
The case~(c) of Theorem~\ref{thm:main} was needed to complete the  combinatorial proof described in~\cite{OB-AM:constellations} of Jackson's formula~\cite{DMJ}. Before embarking on the proof of Theorem~\ref{thm:main}, we make a few remarks.\\

\noindent \textbf{Remark 1.} Theorem~\ref{thm:main} can be stated in terms of uniformly random tuples $\SS$ in $\mS_{\pp,r}$ instead of in $\mM_{\pp,r}$. More precisely, for $\ze\in\{\al,\be,\ga\}$, if $\SS$ is a uniformly random tuple in $\mS_{\pp,r}$, and $f$ is a uniformly random surjection from $[k-1]$ to $[r]$ independent from $\SS$, then the probability for the graph $G_\ze(\SS,f)$ to be a tree is $\frac{|\mM_{\pp,r}|}{|\mS_{\pp,r}|}P_\ze(\pp,r)$. This is simply because the graph $G_\ze(\SS,\vp)$ is never a tree if $\SS\in\mS_{\pp,r}\setminus \mM_{\pp,r}$.\\

\noindent \textbf{Remark 2.} In the case $r=k-1$, the result of Theorem~\ref{thm:main}, can be stated without referring to a random surjection $f$. Indeed, for $\SS\in \mM_{\pp,r}$ and $\ze\in\{\al,\be,\ga\}$, let us define $G_\ze(\SS)$ as the graph $G_\ze(\SS,\Id)$ where $\Id$ is the identity mapping from $[k-1]$ to $[k-1]$. Then the probability $P_\ze(\pp,k-1)$ that the graph $G_\ze(\SS)$ is a tree has the same expression as in Theorem \ref{thm:main}. For instance, the probability that $G_\al(\SS)$ is a tree is $1-p_k/(k-1)$, as claimed in Section \ref{sec:intro}. Indeed, in the particular case $r=k-1$ of Theorem~\ref{thm:main}, the surjection $f$ would be a bijection from $[k-1]$ to $[k-1]$ independent from $\SS$. But then the tuple $(S_{f(1)},\ldots,S_{f(k-1)})$ has the same distribution as $\SS=(S_1,\ldots,S_{k-1})$, hence the surjection $f$ does not affect probabilities.\\

\noindent \textbf{Remark 3.} The results in Theorem~\ref{thm:main} would hold for any probability distribution on the tuples $\SS=(S_1,\ldots,S_k)$ of proper subsets of $[k]$, provided that the probability of a tuple only depends on the total number of occurrences of each integer $i\in[k]$. For instance, the probability that $G_\al(\SS,\vp)$ is a tree would still be equal to the probability that $k\notin S_1$ for such a probability distribution. This result follows from Theorem~\ref{thm:main} since one can always condition on the total number of occurrences of each integer $i\in[k]$.\\

\noindent \textbf{Remark 4.} A slightly weaker version of Theorem~\ref{thm:main} can be obtained by not requiring the function $f$ to be surjective. More precisely, for $\ze\in\{\al,\be,\ga\}$ and for any positive integer $r$, if $\SS$ is a uniformly random tuple in $\mM_{\pp,r}$, and $f$ is a uniformly random function from $[k-1]$ to $[r]$ independent from $\SS$, then the probability $P_\ze(\pp,r)$ that $G_\ze(\SS,f)$ has the same expression as in Theorem \ref{thm:main}. For instance, the probability that $G_\al(\SS,f)$ is a tree is $1-p_k/r$. Indeed, this result follows from Theorem~\ref{thm:main} by conditioning on the cardinality of the image $\Im(f)$ of the function $f$, and by the number of occurrences of each integer $i\in[k]$ in the subsets $(S_{j})_{j\in \Im(f)}$. It is actually this version of Theorem~\ref{thm:main} (in the case $\ze=\ga$) which was needed in~\cite{OB-AM:constellations}. \\


\section{A probabilistic analogue of the matrix-tree theorem}\label{sec:matrix-tree}
The \emph{matrix-tree theorem} is a classical result giving the number of oriented spanning trees of a graph as a determinant; see e.g.~\cite{EC2}. In order to prove Theorem~\ref{thm:main}, it is tempting to consider the trees on the vertex set $[k]$ as the spanning trees of the complete graph $K_k$, and apply a suitable analogue of the matrix-tree theorem. In this section we develop the framework necessary to establish this suitable analogue.

We first recall the matrix-tree theorem (in its \emph{weighted, directed} version). Let $K_n$ denote the \emph{complete digraph} having vertex set $[n]$ and arc set $A=\{(i,j),~i\neq j\}$. We call \emph{spanning tree of $K_n$ rooted at $n$} a set of arcs $T\subseteq A$ not containing any cycle and such that every vertex $i\in[n-1]$ is incident to exactly one outgoing arc in $T$ (this is equivalent to asking that $T$ is a spanning tree of $K_n$ ``oriented toward'' the root vertex $n$). We denote by $\mT_n$ the set of spanning tree of $K_n$ rooted at $n$.
Given some \emph{weights} $w(i,j)$ (taken in a commutative ring) for the arcs $(i,j)\in A$, one defines the \emph{weight} of a tree $T\in \mT_n$ as $\ds w(T)=\prod_{(i,j)\in T}w(i,j)$. The \emph{matrix-tree theorem} states that 
\begin{equation}\label{eq:usual-matrix-tree}
\sum_{T\in \mT_n}w(T)=\det(L),
\end{equation}
where  $L=(L_{i,j})_{i,j\in [n-1]}$ is the \emph{reduced Laplacian matrix}, defined by $L_{i,j}=-w(i,j)$ if $i\neq j$ and $L_{i,i}=\sum_{j\in [n]\setminus \{i\}}w(i,j)$.
Observe that the matrix-tree theorem gives results about the spanning trees of any digraph $G$ with vertex set $[n]$, because one can restrict its attention to the spanning trees of $G$ simply by setting $w(i,j)=0$ for all arcs $(i,j)$ not in $G$.  Observe also that the weights $w(n,j)$ are actually irrelevant in~\eqref{eq:usual-matrix-tree}. 



We will now derive a generalization of the matrix-tree theorem. Let $(\Om,\mA,\Pr)$ be a probability space, where $\Om$ is the sample set, $\mA$ is the set of events (which is a $\si$-algebra on $\Om$), and $\Pr$ is the probability measure. In our applications, $\Om$ will be a finite set and $\mA$ will be the powerset $\mP(\Om)$. Let $n$ be a positive integer and let $E=(E_{i,j})_{i\in[n-1],j\in[n]}$ be a matrix whose entries $E_{i,j}\in \mA$ are events.  For a tree $T\in \mT_n$, we define the \emph{probability of $T$} as
\begin{equation}\nonumber
\Pr(T):=\Pr\Big(\bigcap_{(i,j)\in T}E_{i,j}\Big).
\end{equation}


Now we aim at expressing $\sum_{T\in \mT_n}\Pr(T)$ as some kind of determinant, and this requires some notation. Let $\CC[\mA]$ be the set of (formal) finite linear combinations of events, with coefficients in the field $\CC$ (i.e., the free $\CC$-module with basis $\mA$). The elements of $\CC[\mA]$ are called \emph{generalized events} and are of the form $\la_1 A_1+\cdots +\la_s A_s$ with $\la_1,\ldots,\la_s\in\CC$ and $A_1,\ldots,A_s\in \mA$.
We then define the ring $R=(\CC[\mA],+,\cap)$. Here the intersection operation ``$\cap$'' is defined to act distributively on $\CC[\mA]$, that is, for all $A,B,C\in \mA$, we set $A\cap (B + C)=(A\cap B)+ (A\cap C)$, and moreover for all $\la\in \CC$ we set $A\cap (\la \cdot B)=\la\cdot (A\cap B)$. 
We define the \emph{determinant} of a matrix of generalized events $M=(M_{i,j})_{i,j\in [n]}$ as
$$\det(M)=\sum_{\pi\in \fS_{n}}\eps(\pi)\,M_{1,\pi(1)}\cap M_{2,\pi(2)}\cap \cdots \cap M_{n,\pi(n)},$$
where $\fS_{n}$ denotes the set of permutations of $[n]$, and $\eps(\pi)$ is the sign of the permutation~$\pi$.

We also extend the probability measure $\Pr$ from $\mA$ to $\CC[\mA]$ by linearity. More concretely, we set $\Pr(\la \cdot A)=\la\cdot \Pr(C)$ and $\Pr(A+B)=\Pr(A)+\Pr(B)$. We call \emph{$\Pr$-determinant} of a  matrix of generalized events $M=(M_{i,j})_{i,j\in [n]}$, and denote by $\Pdet(M)$, the probability of $\det(M)$, that is, 
$$\Pdet(M)=\sum_{\pi\in \fS_{n}}\eps(\pi)\,\Pr\l(M_{1,\pi(1)}\cap M_{2,\pi(2)}\cap \cdots \cap M_{n,\pi(n)}\r)$$

We can now state our generalization of the matrix-tree theorem. For a matrix  $E=(E_{i,j})_{i\in[n-1],j\in[n]}$  of generalized events, we define its \emph{reduced Laplacian matrix} $L=(L_{i,j})_{i,j\in[n-1]}$ by setting for all $i\in[n-1]$ $L_{i,j}=-E_{i,j}$ if $i\neq j$ and $L_{i,i}=\sum_{j\neq i}E_{i,j}$. 
\begin{prop}\label{prop:matrix-tree}
Let $E=\l(E_{i,j}\r)_{i\in[n-1],j\in[n]}$ be a matrix of generalized events and let $L=\l(L_{i,j}\r)_{i,j\in [n-1]}$ be its reduced Laplacian matrix. Then 
$$\sum_{T\in \mT_n}\Pr(T)=\Pdet(L).$$
\end{prop}
Observe that if the events $E_{i,j}$ are all independent, then $\ds \Pr(T)=\prod_{(i,j)\in T}\Pr(E_{i,j})$, and $\ds \Pdet(L)=\det\l(\Pr(L_{i,j})\r)_{i,j\in[n-1]}$ so that Proposition~\ref{prop:matrix-tree} reduces to the usual matrix-tree theorem given in~\eqref{eq:usual-matrix-tree} for the weights  $w(i,j)=\Pr(E_{i,j})$.

\begin{proof} 
The known combinatorial proofs of the matrix-tree theorem actually extend almost verbatim to give Proposition~\ref{prop:matrix-tree}. We sketch one such proof, following~\cite{Zeilberger:combinatorial-matrix}, mainly for the reader's convenience. 

By definition,
\begin{eqnarray}
\det(L)&=&\sum_{\pi\in\fS_{n-1}}\eps(\pi)\bigcap_{i\in[n]} L_{i,\pi(i)}\nonumber\\
&=&\sum_{\pi\in\fS_{n-1}}\eps(\pi)\l(\bigcap_{i\in[n],~\pi(i)\neq i} -E_{i,\pi(i)}\r)\cap \l(\bigcap_{i\in[n],~\pi(i)=i}~\sum_{j\neq i} E_{i,j}\r).\label{eq:RHSdet}
\end{eqnarray}
Expanding the right-hand side of~\eqref{eq:RHSdet} leads to a sum over a set that we now describe. Let $B_n$ be the set of digraphs  with vertex set $[n]$ having exactly one outgoing (non-loop) arc at each vertex $i\neq n$, and no outgoing arc at vertex $n$. Let $C_n$ be the set of edge-colored digraphs that can be obtained from digraphs in $B_n$ by coloring edges in either  in \emph{blue} or \emph{red}, in such a way that the blue edges form a disjoint union of simple directed cycles.
We claim that 
\begin{eqnarray}
\det(L)=\sum_{C\in C_n}(-1)^{\cyc(C)}\bigcap_{(i,j)\in C} E_{i,j},\nonumber
\end{eqnarray}
where $\cyc(C)$ is the number of blue cycles of the colored digraph $C$. Indeed, the blue arcs of an element $C\in C_n$ encode a permutation $\pi$ (the blue arcs are $\{(i,\pi(i)),~\textrm{ for all } i  \neq \pi(i)\}$),  the red arcs of $C$ encode a summand in the expansion of $\ds \bigcap_{i\in[n],i=\pi(i)}\sum_{j\neq i} E_{i,j}$ (the red arcs form a set of the form $\{(i,j),~\textrm{ for all }  i=\pi(i) \textrm{ with } j \neq i\}$), and the factor $(-1)^{\cyc(C)}$ is equal to $\ds \eps(\pi)\cdot(-1)^{\#\{i,~ \pi(i)\neq i\}}$.

Now, the digraphs in $B_n$ are all the graphs made of a (possibly empty) tree oriented toward the vertex $n$ together with a (possibly empty) set of directed cycles on which are possibly attached oriented trees. Moreover, if one sums the contribution of all the elements $C\in C_n$ corresponding to the same underlying graph $B\in B_n$ one gets 0 if there are some directed cycles (because these cycles can be colored either blue or red), and $\ds \cap_{(i,j)\in B} E_{i,j}$ otherwise (because all the edges have to be red).
This gives,
$$\det(L)=\sum_{C\in C_n}(-1)^{\cyc(C)}\bigcap_{(i,j)\in C} E_{i,j}=\sum_{C\in T_n}\bigcap_{(i,j)\in C} E_{i,j}.$$
and taking probability on both sides gives $\ds\Pdet(L)=\sum_{C\in T_n}\Pr(C)$.
\end{proof}

\noindent \textbf{Remark.} Observe that more generally, for any commutative ring $(R,+,\otimes)$, any abelian group $(C,+)$, and any homomorphism $P$ from $(R,+)$ to $(C,+)$, there is an analogue of the matrix-tree theorem which holds with the same proof: 
$$\sum_{T\in \mT_n}P(\bigotimes_{(i,j)\in T}E_{i,j})=P(\det(L)).$$

Before closing this section we define an equivalence relation on the set $\CC[A]$ of generalized events. 
Let $A_1,\ldots,A_m$, $B_1,\ldots,B_n$ be events in $\mA$, and let $\la_1,\ldots,\la_m$, $\nu_1\ldots,\nu_n$ be complex numbers. We say that the generalized event $A=\sum_{i=1}^m \la_i A_i$ and $B=\sum_{i=1}^n \nu_i B_i$ are \emph{equivalent}, and we denote this $A\sim B$, if for all $\om\in\Om$, the quantities $A_\om=\sum_{i=1}^m \la_i\textbf{1}_{\om\in A_i}$ and $B_\om=\sum_{i=1}^n \nu_i\textbf{1}_{\om\in B_i}$ are equal. For instance, for all $E,F\in\mA$, the generalized events $E\cup F$ and $E+F-E\cap F$ are equivalent. Also the event $\emptyset$ and the generalized event 0 are equivalent. It is easy to see that $\sim$ is an equivalence relation (symmetric, reflexive, transitive) and that if $E\sim F$ then $\Pr(E)=\Pr(F)$. Moreover, if $E\sim F$ and $E'\sim F'$ then $\la\cdot E \sim \la\cdot F$, $E+F\sim E'+F'$ and $E\cap F\sim E'\cap F'$. We say that two matrices of generalized events $E=(E_{i,j})_{i,j\in [n]}$ and $F=(F_{i,j})_{i,j\in [n]}$ are \emph{equivalent} if $E_{i,j}\sim F_{i,j}$ for all $i\in[n],j\in[n]$. The preceding properties immediately imply the following result.
\begin{lemma}\label{lem:equiv-Pdet}
If $E$ and $F$ are equivalent matrices of events, then they have the same $\Pr$-determinant. 
\end{lemma}

\medskip


\section{Determinantal expressions for the probabilities $P_\ze(\pp,r)$.} \label{sec:determinantal-expressions}
We fix $r$, $k$, and $\pp$ as in Theorem~\ref{thm:main}. We define a probability space $(\Om,\mA,\Pr)$ in the following way:
\begin{compactitem}
\item $\Om$ is the set of pairs $(\SS,f)$, where $\SS$ is in $\mS_{\pp,r}$ and $f$ is a surjection from $[k-1]$ to $[r]$. 
\item $\mA$ is the power set $\mP(\Om)$,
\item $\Pr$ is the uniform distribution on $\Om$.
\end{compactitem}
We denote by $\Om_n$ the set of triples $(\SS,f,\pi)$, where $(\SS,f)$ is in $\Om$ and $\pi$ is a permutation of $[n]$. For a matrix of events $E=(E_{i,j})_{i,j\in[n]}$ we denote by $\Om(E)$ the set of triples $(\SS,f,\pi)\in\Om_n$ such that $(\SS,f)$ is in the intersection $\bigcap_{i=1}^{n} E_{i,\pi(i)}$. Observe that 
\begin{equation}\label{eq:expand-Pdet}
\Pdet(E)=\frac{1}{|\Om|}\sum_{(\SS,f,\pi)\in \Om(E)}\eps(\pi).
\end{equation}

We now express the probabilities $P_\ze(\pp,r)$ defined in Theorem~\ref{thm:main} as $\Pr$-determinants. By definition, for $\ze\in\{\al,\be,\ga\}$, $P_\ze(\pp,r)$ is the conditional probability, in the space $(\Om,\mA,\Pr)$, that the random digraph $G_\ze(\SS,f)$ is a tree given that $\SS=(S_1,\ldots,S_r)$ is in $\mM_{\pp,r}$ (equivalently, given that none of the subsets $S_1,\ldots,S_r$ is equal to $[k]$).
Since the random digraph $G_\ze(\SS,f)$ is never a tree unless $\SS$ is in $\mM_{\pp,r}$ (because $G_\ze(\SS,f)$ has loops if $\SS\notin\mM_{\pp,r}$) one gets 
\begin{eqnarray*}
P_{\ze}(\pp,r)~=~ \frac{\Pr\l(G_\ze(\SS,f)\textrm{ is a tree}\r)}{\Pr\l(\SS\in\mM_{\pp,r}\r)}
~=~\frac{|\mS_{\pp,r}|}{|\mM_{\pp,r}|}\sum_{T\in\mT_k}\Pr\l(G_\ze(\SS,f)=T\r),
\end{eqnarray*}
where  $\mT_k$ is the set of spanning tree of $K_k$ rooted at $k$.

We will now use our generalization of the matrix-tree theorem. For  $\ze$ in $\{\al,\be,\ga\}$, $i\in[k-1]$, and  $j\in[k]$, we define the event $E_{\ze,i,j}\subseteq \Om$ as the set of pairs $(\SS,f)$ in $\Om$ such that $\ze(i,S_{f(i)})=j$. In other words, $E_{\ze,i,j}$ is the event ``the arc $a_i$ of the digraph $G_\ze(\SS,f)$ is $(i,j)$''. By definition, for any tree $T$ in $\mT_k$, the event $G_\ze(\SS,f)=T$ is equal to $\bigcap_{(i,j)\in T}E_{\ze,i,j}$. Thus, 
$$P_\ze(\pp,r)=\frac{|\mS_{\pp,r}|}{|\mM_{\pp,r}|}\sum_{T\in\mT_k}\Pr{\Big(}\bigcap_{(i,j)\in T}E_{\ze,i,j}{\Big )}.$$
Hence by Proposition~\ref{prop:matrix-tree}, 
\begin{eqnarray*}
P_\ze(\pp,r)=\frac{|\mS_{\pp,r}|}{|\mM_{\pp,r}|}\Pdet(L_\ze),
\end{eqnarray*}
where $L_\ze=(L_{\ze,i,j})_{i,j\in[k-1]}$ is the reduced Laplacian matrix of $E_\ze=(E_{\ze,i,j})_{i\in[k-1],j\in[k]}$.

We will now define a matrix $L_\ze'$ equivalent to the reduced Laplacian $L_\ze$.
For $t\in[r]$, and $i,j$ in $[k]$ we define the event $I_{i,j}^{t}$  as follows: 
$$I_{i,j}^{t}=\{(\SS,f)\in\Om,~]i,j]\subseteq S_{t}\}.$$
Observe that $I_{i,i}^{t}:=\Om$ because $]i,i]=\emptyset$. We also define the event $J_{i,j}^{t}$ by $\ds J_{i,j}^{t}=I_{i,j}^{t}$ if $i\neq j$ and 
$\ds J_{i,i}^{t}:=\{(\SS,f),~S_{t}=[k]\}$.
For $i\in[k-1]$ and $j\in [k]$ we define the following generalized events
\begin{equation}\label{eq:L'}
\begin{array}{rcl}
\ds L'_{\al,i,j}&=&\ds I_{i,j}^{f(i)}-I_{i,j-1}^{f(i)},\\[.5em]
\ds L'_{\be,i,j}&=&\ds J_{i,j+1}^{f(i)}-J_{i,j}^{f(i)},\\[.5em]
\ds L'_{\ga,i,j}&=&\ds J_{i,j+1}^{f(i)}-J_{i,j}^{f(i)}-J_{i-1,j+1}^{f(i)}+J_{i-1,j}^{f(i)}.\\[.5em]
\end{array}
\end{equation}
Here and in the following, we consider the subscripts of the events $I^t_{i,j}$ and $J^t_{i,j}$ cyclically modulo $k$; for instance $J^t_{0,j}$ is understood as $J^t_{k,j}$.

It is easy to check that for $\ze$ in $\{\al,\be,\ga\}$ and for all $i\in[k-1],j\in[k]$ such that $i\neq j$, one has the equivalence of events $L'_{\ze,i,j}\sim L_{\ze,i,j}$. For instance, $L_{\al,i,j}=-E_{\al,i,j}$ where 
$$E_{\al,i,j}=\{i+1,..,j-1\in S_{f(i)}, \textrm{ and } j\notin  S_{f(i)} \}=I_{i,j-1}\setminus I_{i,j},$$
hence $\ds L_{\al,i,j}=-\,I_{i,j-1}\setminus I_{i,j}\sim I_{i,j}-I_{i,j-1}$.
Moreover, for all $i\in [k-1]$, 
$$L_{\ze,i,i}'=-\!\sum_{j\in[k-1]\setminus \{i\}}L'_{\ze,i,j}~\sim~ -\!\sum_{j\in[k-1]\setminus \{i\}}L_{\ze,i,j}= L_{\ze,i,i}.$$ 
Thus, by Lemma~\ref{lem:equiv-Pdet} the matrices $L_{\ze}'$ and $L_{\ze}$ have the same $\Pr$-determinant. 
Our findings so far are summarized in the following lemma. 
\begin{lemma}\label{lem:proba-as-determinant}
For all $\ze\in\{\al,\be,\ga\}$, the probability $P_\ze(\pp,r)$ defined in Theorem~\ref{thm:main} is 
$$P_\ze(\pp,r)=\frac{|\mS_{\pp,r}|}{|\mM_{\pp,r}|}\Pdet(L_\ze'),$$
where $L_{\ze}'=(L'_{\ze,i,j})_{i,j\in[k-1]}$ is the matrix of generalized events defined by~\eqref{eq:L'}.\\
\end{lemma}

Next we derive simpler determinantal expressions for the probabilities $P_\ze(\pp,r)$.
\begin{prop}\label{prop:second-det}
For all $\ze\in\{\al,\be,\ga\}$, the probability $P_\ze(\pp,r)$ defined in Theorem~\ref{thm:main} is 
$$P_\ze(\pp,r)=\frac{|\mS_{\pp,r}|}{|\mM_{\pp,r}|}\Pdet(M_\ze),$$
where $M_\ze=(M_{\ze,i,j})_{i,j\in[k]}$ is the matrix of generalized events defined by 
\begin{equation*}
\begin{array}{rcll}
\ds M_{\al,i,j}&=&\ds I_{i,j}^{f(i)} &\textrm{if } i\in[k-1], \textrm{ and } M_{\al,k,j}=\Om,\\[.5em]
\ds M_{\be,i,j}&=&\ds J_{i,j+1}^{f(i)}&\textrm{if } i\in[k-1], \textrm{ and } M_{\be,k,j}=\Om,\\[.5em]
\ds M_{\ga,i,j}&=&\ds J_{i,j+1}^{f(i)}-J_{i-1,j+1}^{f(i)}&\textrm{if } i\in[k-1], \textrm{ and } M_{\ga,k,j}=\Om.\\[.5em]
\end{array}
\end{equation*}
\end{prop}

\noindent \textbf{Example.} Let us illustrate Proposition~\ref{prop:second-det} in the case $\ze=\al$ and $k=3$. In this case, 
$$
M_\al=
\left(\begin{array}{ccc}
I_{1,1}^{f(1)} & I_{1,2}^{f(1)} & I_{1,3}^{f(1)} \\
I_{2,1}^{f(2)} & I_{2,2}^{f(2)} & I_{2,3}^{f(2)} \\
\Om & \Om & \Om\\
\end{array}\right)=
\left(\begin{array}{ccc}
\Om & 2\in S_{f(1)} & 2,3\in S_{f(1)} \\
1,3\in S_{f(2)} & \Om & 3\in S_{f(2)} \\
\Om & \Om & \Om\\
\end{array}\right),
$$
hence by definition of the $\Pr$-determinant,
\begin{equation}\label{eq:expPropal} 
\begin{array}{rcl}
\ds \Pdet(M_\al)&\!\!=\!\!&\ds 
\Pr\l(\Om\r)-\Pr\l(3\in S_{f(2)}\r)+\Pr\l(2\in S_{f(1)}\cap 3\in S_{f(2)}\r)-\Pr\l(2,3\in S_{f(1)}\r)\\[.5em]
&&\ds -\Pr\l(2\in S_{f(1)}\cap 1,3\in S_{f(2)}\r)+\Pr\l(2,3\in S_{f(1)}\cap 1,3\in S_{f(2)}\r). 
\end{array}
\end{equation}
Proposition~\ref{prop:second-det} asserts that the probability $P_\al(\pp,r)$ that the graph $G_{\al}(\SS,f)$ is a tree is equal to $\frac{|\mS_{\pp,r}|}{|\mM_{\pp,r}|}\Pdet(M_\al)$. We leave as an exercise to prove that the right-hand side of~\eqref{eq:expPropal} is equal to $\frac{|\mM_{\pp,r}|}{|\mS_{\pp,r}|}\times (1-p_3/r)$ as predicted by Theorem~\ref{thm:main}.\\

The rest of this section is devoted to the proof of Proposition~\ref{prop:second-det}. 
We first treat in detail the case $\ze=\al$. 
Given Lemma~\ref{lem:proba-as-determinant} we only need to prove $\Pdet(L_\al')=\Pdet(M_{\al}),$
where $L'_\al=\left(I_{i,j}^{f(i)}-I_{i,j-1}^{f(i)}\right)_{i,j\in [k-1]}$. 
Since $\Pr$-determinants are alternating in the columns of matrices, we can replace the $j$th column of $L_\al'$ by the sum of its $j$ first columns without changing the $\Pr$-determinant. This gives,
$$\Pdet(L_\al')=\Pdet\left(I_{i,j}^{f(i)}-I_{i,k}^{f(i)}\right)_{i,j\in [k-1]}.$$
Next, by  linearity of the $\Pr$-determinant in the rows of the matrix, one gets 
$$\Pdet(L_\al')=\sum_{D\subseteq [k-1]}(-1)^{|D|}\Pdet\left(M^{D}\right),$$
where $M^{D}=\left(M^{D}_{i,j}\right)_{i,j\in [k-1]}$ with $M^{D}_{i,j}=I_{i,k}^{f(i)}$ if $i\in D$ and $M^{D}_{i,j}=I_{i,j}^{f(i)}$ otherwise. 
We now show that only $k$ of the subsets $D$ contribute to the above sum.\\

\begin{lemma}
If $D\subseteq[k-1]$ contains more than one element, then $\Pdet(M^{D})=0$.
\end{lemma}

\begin{proof} We assume that $D$ contains two distinct integers $a$ and $b$ and want to show that $\Pdet(M^{D})=0$. We will use the expression~\eqref{eq:expand-Pdet} of $\Pr$-determinant.
By definition, a triple $(\SS,f,\pi)\in\Om_{k-1}$ is in $\Om(M^{D})$ if and only if for all $i\in D$, $]i,k]\subseteq S_{f(i)}$ and for all $i\in [k]\setminus D$, $]i,\pi(i)]\subseteq S_{f(i)}$. Observe that the above conditions for $i=a$ and $i=b$, namely $]a,k]\subseteq \mS_{f(a)}$ and $]b,k]\subseteq \mS_{f(b)}$, do not depend on the permutation $\pi$. More generally, none of the above conditions is affected by changing the permutation $\pi$ by $\pi\circ (a,b)$, where $(a,b)$ is the transposition of the integers $a$ and $b$. Thus a triple $(\SS,f,\pi)$ is in $\Om(M^{D})$ if and only if $\phi(\SS,f,\pi):=(\SS,f,\pi\circ (a,b))$ is in $\Om(M^{D})$. Thus the mapping $\phi$ is an involution of $\Om(M^{D})$. Moreover, since the involution $\phi$ changes the sign of the permutation $\pi$, we get
$$\Pdet(M^{D})=\frac{1}{|\Om|}\sum_{(\SS,f,\pi)\in \Om(M^{D})}\eps(\pi)=0,$$
as claimed.
\end{proof}

So far we have shown that 
$$\Pdet(L_\al')=\Pdet(M^{\emptyset})-\sum_{a\in [k-1]}\Pdet(M^{\{a\}}).$$
Next, we observe that the set of triples $(\SS,f,\pi)$ in $\Om(M^{\emptyset})$ identifies with the set of triples $(\SS,f,\pi')$ in $\Om(M_\al)$ such that $\pi'(k)=k$. Indeed, the correspondence is simply obtained by replacing the permutation $\pi$ of $[k-1]$ by the permutation $\pi'$ of $[k]$ such that $\pi'(k)=k$, and $\pi'(i)=\pi(i)$ for all $i$ in $[k-1]$. Similarly, for all $a\in[k-1]$, there is a bijection between the set of triples $(\SS,f,\pi)$ in $\Om(M^{\{a\}})$ and the set of triples $(\SS,f,\pi')$ in $\Om(M_\al)$ such that $\pi'(a)=k$. Indeed, the bijection is simply obtained by replacing the permutation $\pi$ of $[k-1]$ by the permutation $\pi'$ of $[k]$ such that $\pi'(a)=k$, $\pi'(k)=\pi(a)$ and $\pi'(i)=\pi(i)$ for all $i\neq a$ in $[k-1]$. 
Observe that this bijection changes the sign of the permutation, hence 
\begin{eqnarray*}
\Pdet(L_\al')&=&\frac{1}{|\Om|}\sum_{(\SS,f,\pi)\in \Om(M^{\emptyset})}\eps(\pi)-\sum_{a\in [k-1]}\sum_{(\SS,f,\pi)\in \Om(M^{\{a\}})}\eps(\pi)\\
&=&\frac{1}{|\Om|}\sum_{(\SS,f,\pi')\in \Om(M_\al)}\eps(\pi)=\Pdet(M_\al).
\end{eqnarray*}
This completes the proof of Proposition~\ref{prop:second-det} in the case $\ze=\al$.

The proof of Proposition~\ref{prop:second-det} in the case $\ze=\be$ (resp. $\ze=\ga$), is exactly the same as the proof given above for $\ze=\al$, except that the events $I_{i,j}^t$ are replaced by $J_{i,j+1}^t$ (resp. $J_{i,j+1}^t-J_{i-1,j+1}^t$).

\medskip

\section{Computing $\Pr$-determinants using sign reversing involutions}\label{sec:computing-determinant}
In this section we complete the proof of Theorem~\ref{thm:main} by computing the $\Pr$-determinant of the matrices $M_\ze$ for $\ze\in\{\al,\be,\ga\}$.

\subsection{Computing the $\Pr$-determinant of the matrix $M_\al$.}\label{sec:detMal}
In this section we compute the $\Pr$-determinant of the matrix $M_\al$. 
\begin{prop}\label{prop:detMal}
The $\Pr$-determinant of the matrix $M_\al$ defined in Proposition~\ref{prop:second-det} is 
$$\Pdet(M_\al)=\frac{|\mM_{\pp,r}|}{|\mS_{\pp,r}|}\times (1-p_k/r).$$
\end{prop}
Note that Proposition~\ref{prop:detMal} together with Proposition~\ref{prop:second-det} prove the case $(a)$ of Theorem~\ref{thm:main}. In order to prove Proposition~\ref{prop:detMal}, we first use the multilinearity of $\Pr$-determinants with respect to rows. 
For $a\in[k]$, we denote by $M^{(a)}=(M^{(a)}_{i,j})_{i,j\in[k]}$ the matrix of generalized events represented in Figure \ref{fig:matrix-M}, that is, $M^{(a)}_{i,j}=M_{\al,i,j}$ if $i\geq a-1$, $M^{(a)}_{i,j}=0$ if $i<a-1$ and $j\neq i$, and $M^{(a)}_{i,i}=I_{i,i}^{f(i)}-J_{i,i}^{f(i)}$ if $i<a-1$. Note that $M^{(1)}=M^{(2)}=M_\al$. 
For $a\in\{2,3,\ldots, k-1\}$, we also denote $N^{(a)}=(N^{(a)}_{i,j})_{i,j\in[k]}$ the matrix of events which is the same as  $M^{(a)}$ except the entry in position $(a-1,a-1)$ is $J_{a-1,a-1}^{f(a-1)}$ (see Figure \ref{fig:matrix-M}).
\begin{figure}
$$
{\footnotesize
\left(
\begin{array}{cccccccc}
I^{f(1)}_{1,1} -J^{f(1)}_{1,1} & 0 & \cdots&\cdots & \cdots &\cdots & \cdots & 0\\
0 &I^{f(2)}_{2,2} -J^{f(2)}_{2,2} & 0 & \cdots& \cdots & \cdots&\cdots& 0 \\
\vdots & &\ddots & & & & & \vdots \\ 
0 & \cdots & 0& I^{f(a-2)}_{a-2,a-2} -J^{f(a-2)}_{a-2,a-2} & 0 & \cdots &\cdots & 0\\
I_{a-1,1}^{f(a-1)} &I_{a-1,2}^{f(a-1)} &\cdots & I_{a-1,a-2}^{f(a-1)} &\textcolor{red}{\bX} & I_{a-1,a}^{f(a-1)} & \cdots & I_{a-1,k}^{f(a-1)} \\
I_{a,1}^{f(a)} & I_{a,2}^{f(a)} &\cdots & I_{a,a-2}^{f(a)} & I_{a,a-1}^{f(a)} & I_{a,a}^{f(a)} & \cdots & I_{a,k}^{f(a)} \\
\vdots &\vdots && \vdots & \vdots & \vdots && \vdots \\ 
I_{k-1,1}^{f(k-1)} & I_{k-1,2}^{f(k-1)} & \cdots & I_{k-1,a-2}^{f(k-1)} & I_{k-1,a-1}^{f(k-1)} & I_{k-1,a}^{f(k-1)} & \cdots & I_{k-1,k}^{f(k-1)} \\ 
\Om &\Om &\cdots &\Om &\Om &\Om &\cdots &\Om \\
\end{array}
\right)
}
$$
\caption{The matrix $M^{(a)}$ (resp. $N^{(a)}$) is the matrix represented above with the  entry $\bX$ in position $(a-1,a-1)$ equal to $\bX=I_{a-1,a-1}^{f(a-1)}$ (resp. $\bX=J_{a-1,a-1}^{f(a-1)}$).} \label{fig:matrix-M}
\end{figure}

By linearity of the $\Pr$-determinant with respect to matrix rows one gets
$$\Pdet(M^{(a)})=\Pdet(N^{(a)})+\Pdet(M^{(a+1)}),$$
 for all $a\in\{2,3,\ldots, k-1\}$. Hence
\begin{equation}\label{eq:decompositionMal}
\Pdet(M_\al)=\Pdet(M^{(2)})=\Pdet(M^{(k)})+\sum_{a=2}^{k-1}\Pdet(N^{(a)}).
\end{equation}
We will now show that $\ds \Pdet(M^{(k)})=\frac{|\mM_{\pp,r}|}{|\mS_{\pp,r}|}\times (1-p_k/r)$, and $\Pdet(N^{(a)})=0$ for all $a\in\{2,3,\ldots, k-1\}$, thereby proving Proposition~\ref{prop:detMal}.

\begin{lemma}\label{lem:Mkal}
The $\Pr$-determinant of the matrix $M^{(k)}$ 
is
$\ds \Pdet(M^{(k)})=\frac{|\mM_{\pp,r}|}{|\mS_{\pp,r}|}\times (1-p_k/r).$
\end{lemma}

\begin{proof}
Since all the entries in the last row of  $M^{(k)}$ are equal to  $\Om$, one gets 
$$\Pdet(M^{(k)})=\sum_{\pi\in\fS_k}\eps(\pi)\,\Pr\left(\bigcap_{i=1}^{k-1}M^{(k)}_{i,\pi(i)}\right).$$
Moreover, 
the only permutations $\pi$ contributing to the sum are the identity $\pi=\Id_k$ and the transposition $\pi=(k-1,k)$, which gives
\begin{eqnarray*}
\Pdet(M^{(k)})&=&\Pr\left(\bigcap_{i=1}^{k-2}\left(I_{i,i}^{f(i)}-J_{i,i}^{f(i)}\right)\cap I_{k-1,k-1}^{f(k-1)}\right)-\Pr\left(\bigcap_{i=1}^{k-2}\left(I_{i,i}^{f(i)}-J_{i,i}^{f(i)}\right)\cap I_{k-1,k}^{f(k-1)}\right)\nonumber\\
&=& \Pr\left(\bigcap_{i=1}^{k-2}\left(I_{i,i}^{f(i)}-J_{i,i}^{f(i)}\right)\cap \left(I_{k-1,k-1}^{f(k-1)}-I_{k-1,k}^{f(k-1)}\right)\right). 
\end{eqnarray*}
By definition, $I_{i,i}^{f(i)}=\Om$, and $J_{i,i}^{f(i)}$ is the event ``$S_{f(i)}=[k]$''. Hence the generalized event $\left(I_{i,i}^{f(i)}-J_{i,i}^{f(i)}\right)$ is equivalent to the event ``$S_{f(i)}\neq [k]$''. Moreover, $I_{k-1,k}^{f(k-1)}$ is the event ``$k\in S_{f(k-1)}$''. Hence the generalized event $\left(I_{k-1,k-1}^{f(k-1)}-I_{k-1,k}^{f(k-1)}\right)$ is equivalent to the event ``$k\notin S_{f(k-1)}$''. Thus $\Pdet(M^{(k)})$ is equal to the probability of the event ``for all $i\in[k-1]$, the subset $S_{f(i)}$ is a proper subset of $[k]$ and $k\notin S_{f(k-1)}$''. And since $f$ is a surjection from $[k-1]$ to $[r]$, this event is ``the subsets $S_1,\ldots,S_r$ are proper subsets of $[k]$  and $k\notin S_{f(k-1)}$''. Thus 
$$\Pdet(M^{(k)})=\Pr(\SS \in \mM_{\pp,r})\times \Pr(k\notin S_{f(k-1)} \mid \SS\in \mM_{\pp,r}).$$
Lastly, the conditional probability $\Pr(k\notin S_{f(k-1)} \mid \SS\in \mM_{\pp,r})$ is equal to $1-p_k/r$ since for any tuple $\SS=(S_1\ldots,S_r)\in\mM_{\pp,r}$ there are exactly $p_k$ of the $r$ subsets $S_1,\ldots,S_r$ containing the integer $k$. 
\end{proof}

It remains to prove that $\Pdet(N^{(a)})=0$ for all $a\in \{2,3,\ldots, k-1\}$.
For $D\subseteq [a-2]$, we denote by $N^{a,D}=(N^{a,D}_{i,j})_{i,j \in [k]}$ the matrix of generalized events defined by $N^{a,D}_{i,j}=N^{(a)}_{i,j}$ if $i\geq a-1$ or $j\neq i$, $N^{a,D}_{i,i}=J_{i,i}^{f(i)}$ if $i\in D$, and  $N^{a,D}_{i,i}=I_{i,i}^{f(i)}$ if $i\in [a-2]\setminus D$. By multilinearity of the $\Pr$-determinant in the rows of the matrix, one gets 
\begin{equation}\label{eq:decomposition-Na}
\Pdet(N^{(a)})=\sum_{D\subseteq [a-2]}(-1)^{|D|}\Pdet(N^{a,D}).
\end{equation}
It now suffices to prove the following lemma:

\begin{lemma}\label{lem:pdet=0}
For all $a\in \{2,3,\ldots, k-1\}$, and for all  $D\subseteq [a-2]$,  $\Pdet(N^{a,D})=0$.
\end{lemma}


\noindent \textbf{Intuition for the upcoming proof of Lemma \ref{lem:pdet=0}.} 
We aim at proving $\Pdet(N^{a,D})=0$ using~\eqref{eq:expand-Pdet}. For this, it suffices to find an involution on the set of triples $(\SS,f,\pi)$ in $\Om(N^{a,D})$ changing the sign of the permutation $\pi$. Now, the hope for finding such an involution is based on the observation that the events appearing in the $(a-1)$th and $a$th rows of the matrix $N^{a,D}$ are very similar: for all $j$ in $[k]$ the events $N^{a,D}_{a-1,j}$ and $N^{a,D}_{a,j}$ are respectively ``$\{a\}\cup \, ]a,j]\subseteq S_{f(a-1)}$'' and ``$]a,j]\subseteq S_{f(a)}$''. This implies that changing the permutation $\pi$ into $\pi'=\pi\circ (a-1,a)$ and the surjection $f$ into $f'=f\circ (a-1,a)$ (where $(a-1,a)$ is the transposition of the integers $a-1$ and $a$) will only change the requirement ``$a\in S_{f(a-1)}$'' into ``$a\in S_{f(a)}$''. Now, this can be achieved by also transferring the element $a$ from $S_{f(a-1)}$ to $S_{f(a)}$; and this is what is done in the proof of Lemma~\ref{lem:X0} below. However, this trick does not work for all triples $(\SS,f,\pi)\in\Om(N^{a,D})$ because the requirement ``$a\in S_{f(a-1)}$'' might still hold for $(\SS,f',\pi')$. Indeed, this happens if there is $b\in[k-1]$ such that $f(b)=f(a-1)$ and $a\in\,]b,\pi(b)]$.
In those cases, one would think about changing $\pi$ into $\pi\circ (a-1,b)$ instead; and this is the intuition behind the proof of Lemmas~\ref{lem:Y0} and~\ref{lem:Z0} below. However, this does not work in every case, and one has to deal with a few pathological cases on the side. 
We mention lastly that there exist matrices of generalized events having a non-zero $\Pr$-determinant but whose $(a-1)$th and $a$th rows coincide with those of $N^{a,D}$.\\

In order to prove Lemma \ref{lem:pdet=0} we need some notation. We fix an integer $a\in \{2,3,\ldots, k-1\}$ and a subset $D\subseteq [a-2]$
We define $\fS_{k,a}$ as the set of permutations $\pi$ of $[k]$ such that $\pi(i)=i$ for all $i\in[a-2]$. We also define some subsets $K_{i,j}$ of $[k]$ by setting $K_{i,i}=\emptyset$ if $i\in [a-2]\setminus D$, $K_{i,i}=[k]$ if $i\in D\cup \{a-1\}$, and $K_{i,j}=\,]i,j]$ if $i\in\{a-1,\ldots,k-1\},j\in[k]$ and $(i,j)\neq (a-1,a-1)$. By definition, a triple $(\SS,f,\pi)\in\Om_k$ is in $\Om(N^{a,D})$ if and only if the permutation $\pi$ is in $\fS_{k,a}$, and for all $i\in[k-1]$ the subset $K_{i,\pi(i)}$ is contained in $S_{f(i)}$.
For $\pi\in \fS_{k,a}$, $f$ a surjection from $[k-1]$ to $[r]$, and $t\in [r]$ we define the subset $R_{f,\pi,t}$ of $[k]$ by 
$$\ds R_{f,\pi,t}=\bigcup_{i\in[k-1],~f(i)=t}K_{i,\pi(i)}.$$ 
Observe that a triple $(\SS,f,\pi)\in\Om_k$ is in $\Om(N^{a,D})$ if and only if the permutation $\pi$ is in $\fS_{k,a}$ and for all $t\in[r]$ the subset $R_{f,\pi,t}$ is contained in $S_t$. Lastly for $j\in[k-1]$ we denote 
$$\ds H_{f,\pi,j}=\bigcup_{i\in[k-1]\setminus \{j\},~f(i)=f(j)}K_{i,\pi(i)},$$ 
so that $R_{f,\pi,f(j)}= H_{f,\pi,j}\cup K_{j,\pi(j)}$.

We now define a partition of the set $\Om(N^{a,D})$, by declaring that a triple $(\SS,f,\pi)\in\Om(N^{a,D})$ is in 
\begin{itemize}
\item $\mW$ if $\ds f(a-1)=f(a)$, 
\item $\mX$ if $\ds f(a-1)\neq f(a)$, $\ds a\notin H_{f,\pi,a-1}$, and $\ds a\notin H_{f,\pi,a}$, 
\item $\mY$ if $\ds f(a-1)\neq f(a)$, $\ds a\notin H_{f,\pi,a-1}$, and $\ds a\in H_{f,\pi,a}$, 
\item $\mZ$ if $\ds f(a-1)\neq f(a)$, $\ds a\in H_{f,\pi,a-1}$.
\end{itemize}

Since $\Om(N^{a,D})=\mW\uplus \mX\uplus \mY\uplus \mZ$, Equation~\eqref{eq:expand-Pdet} gives
$$\Pdet(N^{a,D})=\frac{W+X+Y+Z}{|\Om|},$$
where 
\begin{equation}\label{eq:defXYZ}
W=\!\sum_{(\SS,f,\pi)\in \mW}\!\eps(\pi),~~ X=\!\sum_{(\SS,f,\pi)\in \mX}\!\eps(\pi),~~ Y=\!\sum_{(\SS,f,\pi)\in \mY}\!\eps(\pi),~ \textrm{ and } ~ Z=\!\sum_{(\SS,f,\pi)\in \mZ}\!\eps(\pi).
\end{equation}
We will now show that $W=X=Y=Z=0$.\\ 

\noindent \textbf{Remark.} In the particular case $r=k-1$, the surjection $f$ is a bijection and $H_{f,\pi,j}=\emptyset$ for all $j$. Therefore, in this case $\mX=\Om(N^{a,D})$, while $\mW=\mY=\mZ=\emptyset$.\\


\begin{lemma}\label{lem:W0}
The sum $W$ defined in~\eqref{eq:defXYZ} is equal to 0.
\end{lemma}

\begin{proof}
We consider the mapping $\phi$ defined on $\mW$ by setting $\phi(\SS,f,\pi)=(\SS,f,\pi')$, where $\pi'=\pi\circ (a-1,a)$ (where $(a-1,a)$ is the transposition of the integers $a-1$ and $a$). Since $\phi$ changes the sign of the permutation $\pi$, it suffices to prove that $\phi$ is an involution on $\mW$. Clearly this amounts to proving that $(\SS,f,\pi')\in \mW$.
It is clear that $R_{f,\pi',t}=R_{f,\pi,t}$ for all $t\in[r]\setminus \{f(a)\}$. Moreover, remembering $K_{a-1,j}=\{a\}\cup \,]a,j]$ and $K_{a,j}=\,]a,j]$ for all $j\in [k]$ gives 
$$K_{a-1,\pi(a-1)}\cup K_{a,\pi(a)}=\{a\}\cup \,]a,\pi(a-1)]\cup \,]a,\pi(a)]=K_{a-1,\pi'(a-1)}\cup K_{a,\pi'(a)},$$ 
hence $R_{f,\pi',f(a)}=R_{f,\pi,f(a)}$. Thus, for all $t\in[r]$, $R_{f,\pi',t}=R_{f,\pi,t}\subseteq S_t$, that is, $(\SS,f,\pi')\in\Om(N^{a,D})$. Since $f(a-1)=f(a)$ we have $(\SS,f,\pi')\in\mW$ as wanted.
\end{proof}

\begin{lemma}\label{lem:X0}
The sum $X$ defined in~\eqref{eq:defXYZ} is equal to 0.
\end{lemma}

\begin{proof}
We first define a mapping $\phi$ on $\mX$.
Let $(\SS,f,\pi)$ be in $\mX$. 
Note that $a\in S_{f(a-1)}$ (because $a\in K_{a-1,\pi(a-1)}\subseteq S_{f(a-1)}$).
We now denote by $\SS'=(S_1',\ldots,S_r')$ the tuple obtained from $\SS=(S_1,\ldots,S_r)$ by \emph{exchanging the presence or absence of the integer $a$ between the subsets $S_{f(a-1)}$ and $S_{f(a)}$}. More precisely, $\SS'=\SS$ if $a\in S_{f(a)}$, and otherwise $S'_{f(a)}=S_{f(a)}\cup \{a\}$, $S'_{f(a-1)}=S_{f(a-1)}\setminus \{a\}$, and $S_i'=S_i$ for all $i\notin\{ f(a-1),f(a)\}$.
We now define a mapping $\phi$ on $\mX$ by setting $\phi(\SS,f,\pi)=(\SS',f',\pi')$, where $\SS'$ is defined as above, $f'=f\circ (a-1,a)$, and $\pi'=\pi\circ (a-1,a)$. We observe the following identities: 
\begin{eqnarray}
K_{a-1,\pi'(a-1)}=K_{a,\pi(a)}\cup \{a\},\quad K_{a,\pi'(a)}=K_{a-1,\pi(a-1)}\setminus \{a\},\label{eq:Ka}\\
H_{f',\pi',a-1}=H_{f,\pi,a},\textrm{ and } H_{f',\pi',a}=H_{f,\pi,a-1}.\label{eq:Ha}
\end{eqnarray}
We will now prove that $\phi$ is an involution on $\mX$. Observe that this immediately implies that $X=0$ because $\phi$ changes the sign of the permutation $\pi$.\\

\noindent \textbf{Claim.} The mapping $\phi$ is an involution on $\mX$.\\

\noindent \textbf{Proof of the claim.} Clearly, it suffices to prove that \emph{if $(\SS,f,\pi)$ is in $\mX$ then $(\SS',f',\pi')=\phi(\SS,f,\pi)$ is in $\mX$}. Let $(\SS,f,\pi)\in\mX$. Clearly, $\SS'$ is in $\mS_{\pp,r}$, $f'$ is a surjection from $[k-1]$ to $[r]$, and $\pi'$ is in $\fS_{k,a}$. Hence $\phi(\SS,f,\pi)$ is in $\Om_k$. 
Next, we show that $(\SS',f',\pi')$ is in $\Om(N^{a,D})$, that is, $R_{f',\pi',t}\subseteq S_t'$ for all $t\in[r]$. First, for $t\notin \{f(a-1),f(a)\}$, we have $R_{f',\pi',t}=R_{f,\pi,t}$ and $S'_t=S_t$, hence $R_{f',\pi',t}\subseteq S'_t$. 
For $t=f(a-1)=f'(a)$, Equations~\eqref{eq:Ka} and~\eqref{eq:Ha} give
$$R_{f',\pi',t}=H_{f',\pi',a}\cup K_{a,\pi'(a)}=H_{f,\pi,a-1}\cup K_{a-1,\pi(a-1)}\setminus \{a\}=R_{f,\pi,t}\setminus \{a\}.$$
Moreover, the set $S'_t$ does not differ from $S_t$ except maybe for the presence or absence of $a$, hence $R_{f',\pi',t}\subseteq S'_t$. 
Lastly, for $t=f(a)=f'(a-1)$, Equations~\eqref{eq:Ka} and~\eqref{eq:Ha} give
$$R_{f',\pi',t} =H_{f',\pi',a-1}\cup K_{a-1,\pi'(a-1)}=H_{f,\pi,a}\cup K_{a,\pi(a)}\cup \{a\}=R_{f,\pi,t}\cup \{a\}.$$
Moreover, $S'_t=S_t\cup \{a\}$, hence $R_{f',\pi',t}\subseteq S'_t$. 
This shows that $(\SS',f',\pi')$ is in $\Om(N^{a,D})$. Furthermore, $(\SS',f',\pi')$ is in $\mX$ because of~\eqref{eq:Ha} (and the fact that $(\SS,f,\pi)$ in $\mX$). This completes the proof of the claim, hence  $X=0$.
\end{proof}

\begin{lemma}\label{lem:Y0}
The sum $Y$ defined in~\eqref{eq:defXYZ} is equal to 0.
\end{lemma}

\begin{proof}
We define a partition of $\mY$ by declaring that a triple $(\SS,f,\pi)\in\mY$ is in
\begin{itemize}
\item $\mY_1$ if $\ds H_{f,\pi,a}\cup K_{a,\pi(a)}=[k]$, and $\ds H_{f,\pi,a}\cup K_{a,\pi(k)}=[k]$,
\item $\mY_2$ if $\ds H_{f,\pi,a}\cup K_{a,\pi(a)}=[k]$, and $\ds H_{f,\pi,a}\cup K_{a,\pi(k)}\neq [k]$,
\item $\mY_3$ if $\ds H_{f,\pi,a}\cup K_{a,\pi(a)}\neq [k]$.
\end{itemize}
Next, we define a mapping $\phi$ on $\mY$ by setting $\phi(\SS,f,\pi)=(\SS,f,\pi')$, with 
\begin{itemize}
\item $\pi'=\pi\circ (a,k)$ if $(\SS,f,\pi)\in \mY_1$, 
\item $\pi'=\pi\circ (b,k)$ if $(\SS,f,\pi)\in \mY_2$, 
\item $\pi'=\pi\circ (a,b)$ if $(\SS,f,\pi)\in \mY_3$, 
\end{itemize}
where $b$ is the greatest integer in $[k-1]\setminus \{a\}$ such that $f(b)=f(a)$ and $a\in K_{b,\pi(b)}$ (such an integer $b$ necessarily exists since $a\in H_{f,\pi,a}$). Since the mapping $\phi$ clearly changes the sign of the permutation $\pi$, showing that $\phi$ is an involution on $\mY$ would imply $Y=0$. In the rest of this proof we show that $\phi$ is an involution on $\mY$ and in fact an involution on each of the subsets $\mY_1,\mY_2,\mY_3$. 
\\
\noindent\textbf{Claim 1.} The mapping $\phi$ described above is an involution on $\mY_1$.\\

\noindent\textbf{Proof of Claim 1.} 
It clearly suffices to prove that if $(\SS,f,\pi)$ is in $\mY_1$ then $(\SS,f,\pi')=\phi(\SS,f,\pi)$ is in $\mY_1$. First observe that by definition of $\mY_1$,
$$R_{f,\pi,f(a)}=[k]=R_{f,\pi',f(a)}.$$
Hence $R_{f,\pi,t}=R_{f,\pi',t}$ for all $t\in [r]$. Thus $R_{f,\pi',t}=R_{f,\pi,t}\subseteq S_t$, that is, $(\SS,f,\pi')$ is in $\Om(N^{a,D})$. 
Next, we show that $(\SS,f,\pi')\in\mY$. 
Since $(\SS,f,\pi)$ is in $\mY$ we get $\ds a\notin H_{f,\pi,a}=H_{f,\pi',a-1}$. 
Moreover, $a\in H_{f,\pi',a}$ because $a\in [k]=R_{f,\pi',f(a)}$ but $a\notin K_{a,\pi'(a)}$. Thus $(\SS,f,\pi')\in\mY$. Lastly, $(\SS,f,\pi)\in\mY_1$ clearly implies $(\SS,f,\pi')\in\mY_1$.
This completes the proof of Claim~1.\\

\noindent\textbf{Claim 2.} The mapping $\phi$ is an involution on $\mY_2$.\\

\noindent\textbf{Proof of Claim 2.} 
Let $(\SS,f,\pi)$ be in $\mY_2$. 
We need to prove that $(\SS,f,\pi')=\phi(\SS,f,\pi)$ is in $\mY_2$ and that $b$ is also the greatest integer in $[k-1]\setminus \{a\}$ such that $f(b)=f(a)$ and $a\in K_{b,\pi'(b)}$. 
First observe that $b>a-2$, otherwise we would have $\pi(b)=b$ and  $K_{b,\pi(b)}$ equals $\emptyset$ or $[k]$ but either is impossible (indeed $K_{b,\pi(b)} \neq \emptyset$ because $a\in K_{b,\pi(b)}$, and $K_{b,\pi(b)} \neq [k]$ because $(\SS,f,\pi)$ is in $\mY_2$), hence in particular $\pi'=\pi\circ(b,k)$ is in $\fS_{a,k}$. 

We now show that $(\SS,f,\pi')$ is in $\Om(N^{a,D})$, that is, $R_{f,\pi',t}\subseteq S_t$ for all $t\in[r]$. 
First, for $t\neq f(a)$, one has $R_{f,\pi',t}=R_{f,\pi,t}\subseteq S_t$. 
Second, for $t=f(a)$, we have
$$[k]=H_{f,\pi,a}\cup K_{a,\pi(a)}=R_{f,\pi,f(a)}\subseteq S_{f(a)}.$$
Hence, $S_{f(a)}=[k]$ and $R_{f,\pi',f(a)}\subseteq S_{f(a)}$. Thus $(\SS,f,\pi')$ is in $\Om(N^{a,D})$.

We now prove that $(\SS,f,\pi')$ is in $\mY_2$. Let $c$ be the unique integer in $[k]\setminus H_{f,\pi,a}$ such that $]c,a]\subseteq H_{f,\pi,a}$ (the integer $c$ necessarily exists since $a\in H_{f,\pi,a}\neq [k]$). We will show
\begin{equation}\label{eq:positionsClaim2}
a\leq \pi(k)<c\leq b<k, \textrm{ and }a\leq \pi(b)<c.
\end{equation}
Since $b\notin [a-2]$, and $b\neq \{a-1,a\}$, we get $a<b<k$. Moreover, by definition of $b$ and $c$, we get $]b,a]\subseteq\,]c,a]$, hence $a< c\leq b$. Furthermore, we know $]c,a]\cup\,]a,\pi(k)]\subseteq H_{f,\pi,a}\cup \,]a,\pi(k)]\neq [k]$, hence $a\leq \pi(k)<c$. 
Lastly, we know $a\in]b,\pi(b)]$ and $]c,b]\cup]b,\pi(b)]\subseteq H_{f,\pi,a}\neq [k]$, hence $a\leq \pi(b)<c$. Thus~\eqref{eq:positionsClaim2} holds. It implies $a\in ]b,\pi(k)]\subset H_{f,\pi',a}$. Thus $(\SS,f,\pi')$ is in $\mY$.
Moreover $c\notin \,]b,\pi(k)]$, hence $c\notin H_{f,\pi',a}$. Thus $H_{f,\pi',a}\neq [k]$. 
We also know $H_{f,\pi,a}\cup\,]a,\pi(a)]=[k]$, hence $c\in\,]a,\pi(a)]$, which implies $]a,c]\subseteq ]a,\pi(a)]$. 
Moreover, $]c,b]\subseteq]c,a]\subseteq H_{f,\pi,a}$ and $\pi(b)<c$, hence 
$\ds ]c,b]\subseteq \bigcup_{i\in[k-1]\setminus \{a,b\},~f(i)=f(a)}K_{i,\pi(i)}$.
Hence, 
$$ H_{f,\pi',a} \cup]a,\pi(a)]\supseteq \,]c,b]\cup\,]b,a]\cup\,]a,c]=[k].$$ 
Thus $(\SS,f,\pi')$ is in $\mY_2$.
Lastly, we know $a\in\,]b,\pi(k)]=K_{b,\pi'(b)}$, hence $b$ is the greatest integer in $[k-1]\setminus \{a\}$ such that $f(b)=f(a)$ and $a\in K_{b,\pi'(b)}$. This shows that $\phi$ is an involution on $\mY_2$ and completes the proof of Claim~2.


\medskip

\noindent\textbf{Claim 3.} The mapping $\phi$ is an involution on $\mY_3$.\\

\noindent\textbf{Proof of Claim 3.} 
Let $(\SS,f,\pi)$ be in $\mY_3$. We need to prove that $(\SS,f,\pi')=\phi(\SS,f,\pi)$ is in $\mY_3$ and that $b$ is the greatest integer in $[k-1]\setminus \{a\}$ such that $f(b)=f(a)$ and $a\in K_{b,\pi'(b)}$. 
First observe that $b\notin [a-2]$ (indeed $K_{b,\pi(b)} \neq \emptyset$ because $a\in K_{b,\pi(b)}$ and $K_{b,\pi(b)} \neq [k]$ because $(\SS,f,\pi)$ is in $\mY_3$), hence $\pi'=\pi\circ(a,b)$ is in $\fS_{a,k}$. We now show
\begin{equation}\label{eq:positionsClaim3}
a\leq \pi(a)<b<k, \textrm{ and }a\leq \pi(b)<b.
\end{equation}
Since $b\notin [a-2]$, and $b\neq \{a-1,a\}$, we get $a<b<k$. Moreover, we know $a\in ]b,\pi(b)]$, hence $a\leq \pi(b)<b$. 
Moreover, we know $]a,\pi(a)]\cup \,]b,\pi(b)]\subseteq H_{f,\pi,a}\cup \,]a,\pi(a)]\neq [k]$, hence $a\leq \pi(a)<b$. Thus~\eqref{eq:positionsClaim3} holds. It implies 
\begin{equation}\label{eq:positionsClaim3bis}
]a,\pi'(a)]\cup \,]b,\pi'(b)]=\,]b, \max(\pi(a),\pi(b))]=\,]a,\pi(a)]\cup \,]b,\pi(b)]
\end{equation}
hence $R_{f,\pi',f(a)}=R_{f,\pi,f(a)}$. Thus, for all $t\in[r]$, $R_{f,\pi',t}=R_{f,\pi,t}\subseteq S_t$, that is, $(\SS,f,\pi')$ is in $\Om(N^{a,D})$.

Equation~\eqref{eq:positionsClaim3} also implies $a\in \,]b,\pi(a)]\subseteq H_{f,\pi',a}$. Thus $(\SS,f,\pi')$ is in $\mY$.
Moreover, \eqref{eq:positionsClaim3bis} gives $H_{f,\pi',a}\cup \,]a,\pi'(a)]=H_{f,\pi,a}\cup \,]a,\pi(a)]\neq [k]$. Thus $(\SS,f,\pi')$ is in $\mY_3$. Lastly, $a\in\,]b,\pi(a)]=K_{b,\pi'(b)}$, hence $b$ is the greatest integer in $[k-1]\setminus \{a\}$ such that $f(b)=f(a)$ and $a\in K_{b,\pi'(b)}$. This shows that $\phi$ is an involution on $\mY_3$ and completes the proof of Claim~3.

Claims~1,2,3 imply Lemma~\ref{lem:Y0}.
\end{proof}

\begin{lemma}\label{lem:Z0}
The sum $Z$ defined in~\eqref{eq:defXYZ} is equal to 0.
\end{lemma}

\begin{proof}
The proof of Lemma~\ref{lem:Z0} is very similar to the proof of Lemma~\ref{lem:Y0} (actually, it is identical except ``$a$'' is replaced by ``$a-1$'' in certain places). 
We first define a partition of $\mZ$ by declaring that a triple $(\SS,f,\pi)\in\mZ$ is in
\begin{itemize}
\item $\mZ_1$ if $\ds H_{f,\pi,a-1}\cup K_{a-1,\pi(a-1)}=[k]$, and $\ds H_{f,\pi,a-1}\cup K_{a-1,\pi(k)}=[k]$,
\item $\mZ_2$ if $\ds H_{f,\pi,a-1}\cup K_{a-1,\pi(a-1)}=[k]$, and $\ds H_{f,\pi,a-1}\cup K_{a-1,\pi(k)}\neq [k]$,
\item $\mZ_3$ if $\ds H_{f,\pi,a-1}\cup K_{a-1,\pi(a-1)}\neq [k]$.
\end{itemize}
Next, we define a mapping $\phi$ on $\mZ$ by setting $\phi(\SS,f,\pi)=(\SS,f,\pi')$, with 
\begin{itemize}
\item $\pi'=\pi\circ (a-1,k)$ if $(\SS,f,\pi)\in \mZ_1$, 
\item $\pi'=\pi\circ (b,k)$ if $(\SS,f,\pi)\in \mZ_2$, 
\item $\pi'=\pi\circ (a-1,b)$ if $(\SS,f,\pi)\in \mZ_3$, 
\end{itemize}
where $b$ is the greatest integer in $[k-1]\setminus \{a-1\}$ such that $f(b)=f(a-1)$ and $a\in K_{b,\pi(b)}$ (such an integer $b$ necessarily exists since $a\in H_{f,\pi,a-1}$). Since the mapping $\phi$ clearly changes the sign of the permutation $\pi$, showing that $\phi$ is an involution on $\mZ$ would imply $Z=0$. Actually, the proof that $\phi$ is an involution on each of the subsets $\mZ_1$, $\mZ_2$, and $\mZ_3$ is almost identical to the one of Lemma~\ref{lem:Y0} (for instance, the identities~\eqref{eq:positionsClaim2} and~\eqref{eq:positionsClaim3} still hold with $c$ the integer in $[k] \setminus H_{f,\pi,a-1}$ such that $]c,a]\subseteq H_{f,\pi,a-1}$) and is left to the reader. 
\end{proof}

Lemmas~\ref{lem:X0},~\ref{lem:Y0},~\ref{lem:Z0} imply Lemma \ref{lem:pdet=0}. Thus, by~\eqref{eq:decomposition-Na}, $\Pdet(N^{(a)})=0$ for all $a\in \{2,3,\ldots, k-1\}$. This together with Equation~\eqref{eq:decompositionMal} and Lemma~\ref{lem:Mkal} complete the proof of Proposition~\ref{prop:detMal}, hence of the case~$(a)$ of Theorem~\ref{thm:main}.

\subsection{Computing the $\Pr$-determinant of the matrix $M_\be$.}\label{sec:detMbe}
In this section we compute the $\Pr$-determinant of the matrix $M_\be$.

\begin{prop}\label{prop:detMbe}
The $\Pr$-determinant of the matrix $M_\be$ defined in Proposition~\ref{prop:second-det} is 
$$\Pdet(M_\be)=\frac{|\mM_{\qq,r-1}|}{|\mS_{\pp,r}|},$$
where $\qq=(p_1,p_2-1,p_3-1,\ldots,p_k-1)$. 
\end{prop}

Note that Proposition~\ref{prop:detMbe} together with Proposition~\ref{prop:second-det} prove the case $(b)$ of Theorem~\ref{thm:main}. We now sketch the proof of Proposition~\ref{prop:detMbe} which is very similar to the proof of Proposition~\ref{prop:detMal}. We first use some column and row operations on the matrix $M_\be$. Let $M_\be'=(M'_{\be,i,j})_{i,j\in[k]}$, where $M'_{\be,i,j}=J_{i,j}^{f(i)}$ if $i\neq k$ and $M'_{\be,k,j}=\Om$. Since the matrix $M_\be'$ is obtained from $M_\be$ by reordering its columns cyclically we get 
$$\Pdet(M_\be)=(-1)^{k-1}\Pdet(M_\be').$$

For $a\in[k-1]$, we denote by $M^{(a)}=(M^{(a)}_{i,j})_{i,j\in[k]}$ the matrix of generalized events represented in Figure \ref{fig:matrix-Mbeta}, that is,  $M^{(a)}_{i,j}=M_{\be,i,j}'\equiv J_{i,j}^{f(i)}$ if $i\leq a$, $M^{(a)}_{i,j}=0$ if $i\in\{a+1,\ldots,k-1\}$ and $j\neq i$, and $M^{(a)}_{i,i}=I_{i,i}^{f(i)}-J_{i,i}^{f(i)}$ if $i\in\{a+1,\ldots,k-1\}$, and $M^{(a)}_{k,j}=\Om$. Note that $M^{(k-1)}=M_\be'$. 
For $a\in\{2,3,\ldots, k-1\}$, we also denote $N^{(a)}=(N^{(a)}_{i,j})_{i,j\in[k]}$ the matrix of  events which is the same as $M^{(a)}$ except the entry in position $(a,a)$ is $I_{a,a}^{f(a)}$ (see Figure \ref{fig:matrix-Mbeta}).
\begin{figure}
$${\footnotesize \left(
\begin{array}{ccccccc}
J_{1,1}^{f(1)} & \cdots & J_{1,a}^{f(1)} & J_{1,a+1}^{f(1)} & \cdots&\cdots& J_{1,k}^{f(1)} \\ 
\vdots & & \vdots & \vdots &&& \vdots \\ 
J_{a-1,1}^{f(a-1)} & \cdots & J_{a-1,a}^{f(a-1)} & J_{a-1,a+1}^{f(a-1)} & \cdots &\cdots& J_{a-1,k}^{f(a-1)} \\
J_{a,1}^{f(a)} & \cdots & \textcolor{red}{\bX} & J_{a,a+1}^{f(a)} & \cdots &\cdots& J_{a,k}^{f(a)} \\
0 & \cdots & 0 & J^{f(a+1)}_{a+1,a+1} -I^{f(a+1)}_{a+1,a+1} & 0 & \cdots & 0\\
\vdots & & & &\ddots & & \vdots\\
0 & \cdots& \cdots & \cdots & 0 & J^{f(k-1)}_{k-1,k-1} -I^{f(k-1)}_{k-1,k-1} &0\\
\Om &\cdots &\Om &\Om &\cdots &\Om &\Om \\
\end{array}
\right)
}
$$
\caption{The matrix $M^{(a)}$ (resp. $N^{(a)}$) is the matrix represented above with the  entry $\bX$ in position $(a,a)$ equal to $\bX=J_{a,a}^{f(a)}$ (resp. $\bX=I_{a,a}^{f(a)}$).} \label{fig:matrix-Mbeta}
\end{figure}

By linearity of the $\Pr$-determinant with respect to matrix rows one gets for all $a\in\{2,3,\ldots, k-1\}$,
$$\Pdet(M^{(a)})=\Pdet(N^{(a)})+\Pdet(M^{(a-1)}).$$
Thus, using this relation iteratively starting with $\Pdet(M_\be')=\Pdet(M^{(k-1)})$ we get
\begin{equation}\label{eq:decompositionMbe}
\Pdet(M_\be')=\Pdet(M^{(k-1)})=\Pdet(M^{(1)})+\sum_{a=2}^{k-1}\Pdet(N^{(a)}).
\end{equation}
We will now prove that $\ds \Pdet(M^{(1)})=(-1)^{k-1}\frac{|\mM_{\qq,r-1}|}{|\mS_{\pp,r}|}$, and $\Pdet(N^{(a)})=0$ for all $a\in\{2,3,\ldots, k-1\}$.

\begin{lemma}\label{lem:Mkbe}
The $\Pr$-determinant of the matrix $M^{(1)}$ appearing in~\eqref{eq:decompositionMbe} is
$\ds \Pdet(M^{(1)})=(-1)^{k-1}\frac{|\mM_{\qq,r-1}|}{|\mS_{\pp,r}|}$.
\end{lemma}

\begin{proof} This proof is similar to the one of Lemma~\ref{lem:Mkal}. 
One easily gets 
\begin{eqnarray*}
\Pdet(M^{(1)})&=&\Pr\left(\bigcap_{i=2}^{k-1}\left(J_{i,i}^{f(i)}-I_{i,i}^{f(i)}\right)\cap \left(J_{1,1}^{f(1)}-J_{1,k}^{f(1)}\right)\right)\\
&=&(-1)^{k-1}\Pr\left(\bigcap_{i=2}^{k-1}\left(I_{i,i}^{f(i)}-J_{i,i}^{f(i)}\right)\cap \left(J_{1,k}^{f(1)}-J_{1,1}^{f(1)}\right)\right).
\end{eqnarray*}
Moreover, the generalized event inside the argument of $\Pr(\cdot)$ is clearly equivalent to the event ``$S_{f(1)}=\{2,3,\ldots,k\}$ and for all $i\in\{2,\ldots,k-2\}$, $S_{f(i)}$ is a proper subset of $[k]$''. This event is the same as ``$S_{f(1)}=\{2,3,\ldots,k\}$ and $\SS\in \mM_{\pp,r}$'' which has probability 
$$\Pr(\SS \in \mM_{\pp,r})\times \Pr(S_{f(1)}=\{2,3,\ldots,k\} \mid \SS\in \mM_{\pp,r})=\frac{|\mM_{\pp,r}|}{|\mS_{\pp,r}|}\times \frac{|\mM_{\qq,r-1}|}{|\mM_{\pp,r}|}=\frac{|\mM_{\qq,r-1}|}{|\mS_{\pp,r}|}.$$ 
\end{proof}

It remains to prove that $\Pdet(N^{(a)})=0$ for all $a\in \{2,3,\ldots, k-1\}$. The proof is very similar to the one presented in Section~\ref{sec:detMal} (indeed we point out that the rows $a-1$ and $a$ of $N^{(a)}$ are identical to the ones in Section~\ref{sec:detMal}). For $D\subseteq \{a+1,\ldots,k-1\}$, we denote by $N^{a,D}=(N^{a,D}_{i,j})$ the matrix of generalized events defined by $N^{a,D}_{i,j}=N^{(a)}_{i,j}$ if $i\leq a$ or $j\neq i$, and $N^{a,D}_{i,i}=I_{i,i}^{f(i)}$ if $i\in D$, and $N^{a,D}_{i,i}=J_{i,i}^{f(i)}$ if $i\in \{a+1,\ldots,k-1\}\setminus D$.
By linearity of the $\Pr$-determinant, one gets 
$$\Pdet(N^{(a)})=\sum_{D\subseteq \{a+1,\ldots,k-1\}}(-1)^{|D|}\Pdet(N^{a,D}).$$

We now fix $a\in \{2,3,\ldots, k-1\}$ and $D\subseteq \{a+1,\ldots,k-1\}$ and proceed to prove that $\Pdet(N^{a,D})=0$.
We define $\fS_{k}^a$ as the set of permutations $\pi$ of $[k]$ such that $\pi(i)=i$ for all $i>a$. We also define some subsets $K_{i,j}$ of $[k]$ by setting $K_{i,i}=[k]$ if $i\in D\cup \{a-1\}$, $K_{i,i}=\emptyset$ if $i\in \{a,a+1,\ldots,k-1\}\setminus D$, and $K_{i,j}=\,\,]i,j]$ for $i\in \{1,2,\ldots,a\}$, 
and $j\in[k]\setminus \{i\}$. By definition, a triple $(\SS,f,\pi)\in\Om_k$ is in $\Om(N^{a,D})$, if and only if the permutation $\pi$ is in $\fS_{k}^a$, and for all $i\in[k-1]$ the subset $K_{i,\pi(i)}$ is contained in $S_{f(i)}$. 
Lastly for $j\in[k-1]$ we denote 
$$\ds H_{f,\pi,j}=\bigcup_{i\in[k-1]\setminus \{j\},~f(i)=f(j)}K_{i,\pi(i)}.$$ 
We now define a partition of the set $\Om(N^{a,D})$, by declaring that a triple $(\SS,f,\pi)\in\Om(N^{a,D})$ is in 
\begin{itemize}
\item $\mW$ if $\ds f(a-1)=f(a)$, 
\item $\mX$ if $\ds f(a-1)\neq f(a)$, $\ds a\notin H_{f,\pi,a-1}$, and $\ds a\notin H_{f,\pi,a}$, 
\item $\mY$ if $\ds f(a-1)\neq f(a)$, $\ds a\notin H_{f,\pi,a-1}$, and $\ds a\in H_{f,\pi,a}$, 
\item $\mZ$ if $\ds f(a-1)\neq f(a)$, $\ds a\in H_{f,\pi,a-1}$.
\end{itemize}
Since $\Om(N^{a,D})=\mW\uplus \mX\uplus \mY\uplus \mZ$, we get $\Pdet(N^{a,D})=\frac{W+X+Y+Z}{|\Om|}$, where 
$$W=\!\sum_{(\SS,f,\pi)\in \mW}\!\eps(\pi),~~ X=\!\sum_{(\SS,f,\pi)\in \mX}\!\eps(\pi),~~ Y=\!\sum_{(\SS,f,\pi)\in \mY}\!\eps(\pi),~ \textrm{ and } ~ Z=\!\sum_{(\SS,f,\pi)\in \mZ}\!\eps(\pi).$$
We then show that $W=X=Y=Z=0$.\\ 

The proof that $W=0$ is identical to the proof of Lemma~\ref{lem:W0}. The proof that $X=0$ is identical to the proof of Lemma~\ref{lem:X0}. It is done by defining a sign reversing involution $\phi$ on $\mX$. Explicitly, the mapping $\phi$ is defined on $\mX$ by setting $\phi(\SS,f,\pi)=(\SS',f',\pi')$, where $f'=f\circ (a-1,a)$, and $\pi'=\pi\circ (a-1,a)$, and $\SS'$ is obtained from $\SS$ simply by exchanging the presence or absence of the integer $a$ between the subsets $S_{f(a-1)}$ and $S_{f(a)}$. The proof that $\phi$ is an involution on $\mX$ is identical to the one given in the proof of Lemma~\ref{lem:X0}. 

The proof that $Y=0$ is identical to the proof of Lemma~\ref{lem:Y0}. It is done by defining a sign reversing involution $\phi$ on $\mY$. Explicitly, one considers the partition $\mY=\mY_1 \uplus \mY_2 \uplus \mY_3$ defined exactly as in the proof of Lemma~\ref{lem:Y0}.
One then defines a mapping $\phi$ on $\mY$ by setting $\phi(\SS,f,\pi)=(\SS,f,\pi')$, where $\pi'=\pi\circ (a,k)$ (resp. $\pi'=\pi\circ (b,k)$, $\pi'=\pi\circ (a,b)$) if $(\SS,f,\pi)$ is in $\mY_1$ (resp. $\mY_2$, $\mY_3$),
where $b$ is the greatest integer in $[k-1]\setminus \{a\}$ such that $f(b)=f(a)$ and $a\in K_{b,\pi(b)}$. The fact that $\phi$ is an involution on $\mY$ (actually, on each of the subsets $\mY_1$, $\mY_2$, and $\mY_3$) is identical to the one given in the proof of Lemma~\ref{lem:Y0}. 

The proof that $Z=0$ is again identical to the proof of Lemma~\ref{lem:Z0}. This completes the proof that $\Pdet(N^{a,D})=0$ for all $a\in \{2,3,\ldots, k-1\}$ and all $D\subseteq \{a+1,\ldots,k-1\}$. This together with Lemma~\ref{lem:Mkbe} complete the proof of Proposition~\ref{prop:detMbe}, hence of the case~$(b)$ of Theorem~\ref{thm:main}.

\subsection{Computing the $\Pr$-determinant of the matrix $M_\ga$.}\label{sec:detMde}
In this section we compute the $\Pr$-determinant of the matrix $M_\ga$.
\begin{lemma}\label{lem:third-det-delta}
The $\Pr$-determinant of the matrix $M_\ga$ defined in Proposition~\ref{prop:second-det} is given by
$$\Pdet(M_\ga)=\sum_{a=1}^{k}\Pdet(Q^{(a)}),$$
where $Q^{(a)}=(Q_{i,j}^{(a)})_{i,j\in[k]}$ is the matrix of generalized events defined by
$\ds Q_{i,j}^{(a)}= J_{i,j+1}^{f(i)}$ if $i\in[a-1]$, $Q_{i,j}^{(a)}= J_{i,j+1}^{f(i-1)}$ if $i\in \{a+1,\ldots,k\}$, and $Q^{(a)}_{a,j}=\Om$. 
\end{lemma}

Note that $Q^{(k)}=M_{\be}$, where $M_\be$ is the matrix defined in Proposition~\ref{prop:second-det}.

\begin{cor}\label{cor:detMde}
The $\Pr$-determinant of the matrix $M_\ga$ defined in Proposition~\ref{prop:second-det} is 
$$\Pdet(M_\ga)=\sum_{a=1}^k\frac{|\mM_{\qq^{(a)},r-1}|}{|\mS_{\pp,r}|},$$
where $\qq^{(a)}=(q_1^{(a)},\ldots,q_k^{(a)})$ and for all $i\in[k]$, $q_i^{(a)}=p_{i}-1$ if $i\neq j$ and $q_a^{(a)}=p_{a}$.
\end{cor}
 
\begin{proof}[Proof of Corollary~\ref{cor:detMde}]
By Proposition~\ref{prop:detMbe}, 
$$ \Pdet(Q^{(k)})\equiv \Pdet(M_\be)=\frac{|\mM_{\qq^{(1)},r-1}|}{|\mS_{\pp,r}|}.$$ 
Now for $a\in[k-1]$, permuting cyclically $k-a$ times both the rows and the columns of the matrix $Q^{(a)}$  gives
$$\Pdet(Q^{(a)})=\Pdet(R^{(a)}),$$
where $R^{(a)}=(R_{i,j}^{(a)})_{i,j\in[k]}$ is the matrix defined by $R_{i,j}^{(a)}=J_{i+a,j+a+1}^{f(i+a-1)}$ for $i\in [k-a]$, $R_{i,j}=J_{i+a,j+a+1}^{f(i+a-k)}$ for $i\in \{k-a+1,\ldots,k-1\}$, and $R^{(a)}_{k,j}=\Om$. Since the surjection $f$ is uniformly random we get, 
$$\Pdet(R^{(a)})=\Pdet(T^{(a)}),$$ 
where  $T^{(a)}=(T_{i,j}^{(a)})_{i,j\in[k]}$ is the matrix defined by $T_{i,j}^{(a)}=J_{i+a,j+a+1}^{f(i)}$ for $i\in [k-1]$ and $T^{(a)}_{k,j}=\Om$. 
Now, by symmetry (obtained by replacing the integer $i$ by $i+a$ in the subsets $S_1,\ldots,S_r$), the $\Pr$-determinant  $\Pdet(T^{(a)})$ is obtained from  $\Pdet(M_\be)$ by replacing $p_i$ by $p_{i+a}$ for all $i\in [k]$, that is, 
$$\Pdet(T^{(a)})=\frac{|\mM_{\qq^{(a+1)},r-1}|}{|\mS_{\pp,r}|}.$$
This together with Lemma~\ref{lem:third-det-delta} completes the proof of Corollary~\ref{cor:detMde}.
\end{proof}

Note that Corollary~\ref{cor:detMde} together with Proposition~\ref{prop:second-det} prove the case $(c)$ of Theorem~\ref{thm:main}. It now only remains to prove Lemma~\ref{lem:third-det-delta}. 
For $D\subseteq [k-1]$ and $j\in [k]$, we denote $M^D=\l(M_{i,j}^D\r)_{i,j\in [k]}$ the matrix of events defined by $M_{i,j}^D=J_{i-1,j+1}^{f(i)}$ if $i\in D$, $M_{i,j}^D=J_{i,j+1}^{f(i)}$ if $i\in[k-1]\setminus D$, and $M_{k,j}^D=\Om$. By the multilinearity of $\Pr$-determinants with respect to the rows,
$$\Pdet(M_\ga)=\sum_{D\subseteq [k-1]}(-1)^{|D|}\Pdet\l(M^D\r).$$
We now show that only $k$ of the subsets $D$ contribute to the above sum.\\

\noindent\textbf{Claim.} If $D\subseteq[k-1]$ contains an integer $a>1$ but not $a-1$, then $\Pdet\l(M^D\r)=0$.\\

\noindent\textbf{Proof of the claim.} We suppose that the subset $D$ contains the integer $a>1$ but not $a-1$. We remark that the $(a-1)$th and $a$th rows of the matrix $M^{D}$ are almost identical: indeed for all $j\in[k]$, $M_{a-1,j}^D=J_{a-1,j+1}^{f(a-1)}$ while $M_{a,j}^D=J_{a-1,j+1}^{f(a)}$. We now show that this implies $\Pdet(M^{D})=0$. We define a mapping $\phi$ on $\Om(M^{D})$ by setting $\phi(\SS,f,\pi)=(\SS,f\circ (a-1,a),\pi\circ (a-1,a))$. We want to prove that $\phi$ is an involution on $\Om(M^{D})$. Let $(\SS,f,\pi)\in\Om_k$. It is clear that $f\circ (a,a-1)$ is a surjection from $[k-1]$ to $[r]$ so that $\phi(\SS,f,\pi)$ is in $\Om_k$. Moreover, it is easy to see from the above remark about the $(a-1)$th and $a$th rows of $M^{D}$ that $\phi(\SS,f,\pi)$ is in $\Om(M^{D})$ (indeed the triple $\phi(\SS,f,\pi)$ satisfies the condition $M_{a-1,\pi(a-1)}^D$ because $(\SS,f,\pi)$ satisfies $M_{a,\pi(a)}^D$, and $\phi(\SS,f,\pi)$ satisfies $M_{a,\pi(a)}^D$ because $(\SS,f,\pi)$ satisfies $M_{a-1,\pi(a-1)}^D$). Thus the mapping $\phi$ is an involution of $\Om(M^{D})$. Moreover, since the involution $\phi$ changes the sign of the permutation $\pi$ we get
$$\Pdet(M^{D})=\frac{1}{|\Om|}\sum_{(\SS,f,\pi)\in \Om(M^{D})}\eps(\pi)=0,$$
as claimed.\\

So far we have proved that 
$$\Pdet(M_\ga)=\Pdet(M^{\emptyset})+\sum_{a=1}^{k-1}(-1)^a\, \Pdet(M^{[a]}).$$
Moreover, by definition $M^{\emptyset}=Q^{(k)}$, and it is easy to see, by reordering the rows of $M^{[a]}$, that for all $a\in [k-1]$ we have $\Pdet\l(M^{[a]}\r)=(-1)^a\,\Pdet(Q^{(a)})$. 
This completes the proof of Lemma~\ref{lem:third-det-delta}, hence of the case $(c)$ of Theorem~\ref{thm:main}.


\section{Open questions and conjectures} \label{sec:conj-gamma}
This paper leaves open a few questions. 
First of all, the proof we obtained of Theorem~\ref{thm:main} is not as elegant as we had hoped for, hence the following question:
\begin{Question}
Is there a more direct proof of Theorem~\ref{thm:main}? In particular, can one prove Theorem~\ref{thm:main} without using the matrix-tree theorem?
\end{Question}
Second, our use of the matrix-tree theorem (for which a ``forest version'' exists) suggests the following question:
\begin{Question}
For $\ze\in\{\al,\be,\de\}$, is there a simple expression for the probability that the random graph $G_\ze(\SS,\vp)$ is a \emph{pseudo-forest}, that is, a digraph in which the only cycles are loops (so that the connected components are trees oriented toward a vertex which has a loop)? 
\end{Question}
More importantly, Theorem~\ref{thm:main} gives simple formulas for the probability that certain random graphs are trees. We stumbled over one of these formulas while studying a factorization problem in the symmetric group \cite{OB-AM:constellations}, but it is unclear how general this phenomenon is.
\begin{Question}
Is there a more general theory encompassing Theorem~\ref{thm:main}?
\end{Question}

We conclude with a few conjectures about a random digraph $G_\de$ defined analogously to $G_\al$, $G_\be$, and $G_\ga$.
Let $\de$ be the mapping taking as argument an integer $i\in [k]$ and a subset $S\subseteq [k]$ and  defined by $\de(i,S)=i$ if $i\in S$, and otherwise $\de(i,S)$ is the integer $j\in[k]\setminus S$ such that $]i,j-1]\subseteq S$. The mapping $\de$ is represented in Figure~\ref{fig:delta}.
\fig{width=5.5cm}{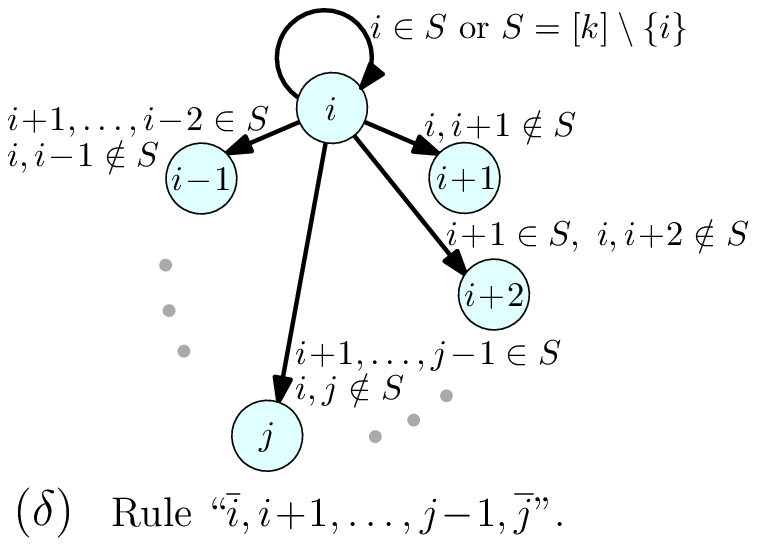}{Rule $\de$ for creating an arc of the complete graph $K_k$.} 

Given a tuple $\SS = (S_1,\ldots,S_r)$ in $\mS_{\pp,r}$ and a surjection $f$ from $[k-1]$ to $[r]$, we define  $G_{\de}(\SS,f)$ 
 to be the digraph with vertex set $[k]$ and arc set $A=\{a_1,\ldots,a_{k-1}\}$ where $a_i=(i,\de(i,S_{\vp(i)}))$ for all $i\in[k-1]$.
The following conjecture suggests there is a simple expression for the probability that $G_{\de}(\SS,f)$ is a tree.
\begin{conjecture} \label{conj:casega}
Let $k$ and $r$ be positive integers such that $r<k$, and let $\pp=(p_1,\ldots,p_k)$ be a tuple of non-negative integers.
Let $\SS=(S_1,\ldots,S_r)$ be a uniformly random element of $\mM_{\pp,r}$ (supposing that this set is non-empty), and let $\vp$ be a uniformly random surjection from $[k-1]$ to $[r]$ independent from~$\SS$. 
 Then the probability $P_\de(\pp,r)$ that the random digraph $G_\de(\SS,\vp)$ is a tree is
$$ P_\de(\pp,r)= \frac{|\mM_{\pp,r-1}|}{|\mM_{\pp,r}|}.$$ 
This is equal to the probability that $S_1=\emptyset$.
\end{conjecture}

Although we were unable to prove Conjecture \ref{conj:casega}, we obtained the following analogue of Proposition~\ref{prop:second-det}:
\begin{eqnarray*}
P_\de(\pp,r)&= \ds \frac{|\mS_{\pp,r}|}{|\mM_{\pp,r}|}\Pdet(M_\de),
\end{eqnarray*}
where $M_{\de} = \l(M_{\de,i,j}\r)_{i,j \in [k]}$ is the matrix of the generalized event defined by  $M_{\de,i,j}=\ds I_{i,j}^{f(i)}-J_{i-1,j}^{f(i)}$ for  $i\in[k-1]$,  and  $M_{\de,k,j}=\Om$.
Thus Conjecture~\ref{conj:casega} is equivalent to the following identity: 
\begin{equation}\label{eq:expressionD}
\Pdet(M_\de)=\ds\frac{|\mM_{\pp,r-1}|}{|\mS_{\pp,r}|}.
\end{equation}

One approach to try to prove \eqref{eq:expressionD} is to use the multilinearity of $\Pr$-determinants.
For $D\subseteq [k-1]$ and $j\in [k]$, we denote $M^D=\l(M_{i,j}^D\r)_{i,j\in [k]}$, where $M_{i,j}^D=J_{i-1,j}^{f(i)}$ if $i\in D$, $M_{i,j}^D=I_{i,j}^{f(i)}$ if $i\in[k-1]\setminus D$, and $M_{k,j}^D=\Om$. By multilinearity of $\Pr$-determinants,
\begin{equation} \label{eq-mult-gamma}
\Pdet(M_\de)=\sum_{D\subseteq [k-1]}(-1)^{|D|}\Pdet\l(M^D\r).
\end{equation}
Calculations suggest the following formula for the $\Pr$-determinant of $M^D$:
\begin{conjecture} \label{conj:pdetMD}
For all $D\subseteq [k-1]$,
$$\Pdet(M^D) = \frac{|\mM_{\pp,r}|}{|\mS_{\pp,r}|} \times \Pr( k \not\in S_1, D\subseteq S_1 \mid \SS \in \mM_{\pp,r}). $$
\end{conjecture}

\noindent \textbf{Remark.}
Conjecture~\ref{conj:pdetMD} implies Conjecture~\ref{conj:casega} since substituting the conjectured formula for $\Pdet(M^D)$ in~\eqref{eq-mult-gamma} and doing inclusion-exclusion gives
$$
\Pr(M_\de) = \frac{|\mM_{\pp,r}|}{|\mS_{\pp,r}|} \times \Pr(S_1=\emptyset \mid \SS \in \mM_{\pp,r}).
$$

\noindent \textbf{Remark.} When $D=\emptyset$ then $M^D=M_{\alpha}$ as defined in Proposition~\ref{prop:second-det}. Hence, in this case Conjecture~\ref{conj:pdetMD} holds by Proposition~\ref{prop:detMal}. Similarly, when $D=[k-1]$ the matrix $M^D$ is the matrix $M_{\beta}$ as defined in Proposition~\ref{prop:second-det}, up to shifting cyclically the $k$ columns one position to the right and shifting the first $k-1$ rows one position down. Hence, in this case Conjecture~\ref{conj:pdetMD} holds by Proposition~\ref{prop:detMbe}.

\bigskip

\noindent \textbf{Acknowledgments.} We thank Omer Angel, Robin Pemantle, and Peter Winkler for interesting discussions.
\bibliographystyle{plain} 
\bibliography{biblio-random-arbo}

\end{document}